\documentclass[11pt]{article}
\usepackage[T1]{fontenc}
\usepackage{lmodern}
\usepackage[a4paper,margin=25mm]{geometry}
\usepackage{amsmath,amssymb,amsthm,mathtools}
\usepackage{mathrsfs}
\usepackage{microtype}
\usepackage{needspace}
\usepackage[colorlinks=true,linkcolor=blue,citecolor=blue,urlcolor=blue]{hyperref}

\numberwithin{equation}{section}
\newtheorem{theorem}{Theorem}[section]
\newtheorem{lemma}[theorem]{Lemma}
\newtheorem{proposition}[theorem]{Proposition}
\newtheorem{corollary}[theorem]{Corollary}
\newcommand{\R}{\mathbb R}
\newcommand{\dd}{\,\mathrm d}
\newcommand{\ip}[2]{\langle #1,#2\rangle}
\newcommand{\norm}[1]{\lVert #1\rVert}
\newcommand{\abs}[1]{\lvert #1\rvert}
\newcommand{\Id}{\mathrm I}
\newcommand{\Span}{\operatorname{span}}
\newcommand{\Sym}{\operatorname{Sym}}
\newcommand{\dist}{\operatorname{dist}}
\newcommand{\HH}{\mathcal H}
\newcommand{\BB}{\mathbb B}
\newcommand{\cE}{\mathcal E}
\newcommand{\cN}{\mathcal N}
\newcommand{\cB}{\mathcal B}
\newcommand{\cL}{\mathcal L}
\newcommand{\cD}{\mathcal D}
\newcommand{\cO}{\mathcal O}
\newcommand{\cT}{\mathcal T}
\newcommand{\cZ}{\mathcal Z}
\newcommand{\cM}{\mathcal M}
\newcommand{\cF}{\mathcal F}
\newcommand{\cW}{\mathscr W}
\newcommand{\cY}{\mathscr Y}
\newcommand{\dS}{\dot S^1}
\newcommand{\cU}{\mathcal U}
\newcommand{\Def}{\mathcal D_G}

\title{Extremals and stability for the\\
	Folland-Stein inequality on H-type groups}

\author{Zhipeng Yang\thanks{Corresponding author: yangzhipeng326@163.com.}}
\date{}

\AtEndDocument{%
	\par
	\bigskip
	\bigskip
	\noindent
	\textbf{Zhipeng Yang:}\\[0.3em]
	\textsc{Department of Mathematics, Yunnan Normal University, Kunming, China}\\[0.3em]
	\textsc{Yunnan Key Laboratory of Modern Analytical Mathematics and Applications, Kunming 650500, China}\\[0.3em]
	\textsc{Department of Mathematics: Analysis, Logic and Discrete Mathematics, Ghent University, Belgium}\\[0.3em]
	\textit{E-mail address}: \texttt{yangzhipeng326@163.com}%
}

\begin{document}
	\maketitle
	\begin{abstract}
		We classify all real extremals of the sharp $L^2$ Folland-Stein inequality on H-type groups and prove a global Bianchi-Egnell stability estimate. The linearized kernel at the standard bubble consists exactly of its translation and dilation derivatives. We also obtain explicit bounds for the spectral gap above the symmetry modes and characterize the optimal coefficient in the local stability asymptotics. The results apply to arbitrary orthogonal Clifford modules.
	\end{abstract}
	
	\section{Introduction and main results}
	
	Let $G$ be a connected, simply connected, nonabelian H-type group. Its metric Lie algebra has an orthogonal decomposition $\mathfrak g=V\oplus Z$, where $[V,V]=Z$ and $Z$ is the center. We write
	\[
	m=\dim V,\qquad n=\dim Z,\qquad Q=m+2n,
	\qquad 2^*=\frac{2Q}{Q-2}.
	\]
	For $t\in Z$, define $J_t:V\to V$ by
	$\ip{J_tx}{y}=\ip{t}{[x,y]}$. The H-type condition is
	$J_t^2=-|t|^2\Id$. Thus, for an orthonormal basis $(T_\alpha)_{\alpha=1}^n$ of $Z$ and $J_\alpha=J_{T_\alpha}$,
	\begin{equation}\label{eq1.1}
		J_\alpha^*=-J_\alpha,\qquad
		J_\alpha J_\beta+J_\beta J_\alpha=-2\delta_{\alpha\beta}\Id.
	\end{equation}
	Kaplan introduced H-type groups in \cite{Kaplan1980}. The nonabelian nilpotent factors in the Iwasawa decompositions of simple rank-one groups form a subclass characterized by the $J^2$ condition \cite{CowlingDooleyKoranyiRicci1991}: for $z\perp z'$ and $x\in V$, there is $z''\in Z$ such that $J_zJ_{z'}x=J_{z''}x$. The results below apply to every orthogonal Clifford module satisfying \eqref{eq1.1}, including reducible modules. For example, take $V=\mathbb H$ with its Euclidean metric and $Z=\R^2$, with $J_1x=\mathbf i x$ and $J_2x=\mathbf j x$ given by left quaternionic multiplication. This gives $m=4$ and $n=2$. At $x=1$, $J_1J_2x=\mathbf k$ is orthogonal to $\Span\{J_1x,J_2x\}$, so this group lies outside the Iwasawa subclass.
	
	We identify $G$ with $V\times Z$ through exponential coordinates and use the group law
	\[
	(x,t)(y,z)=(x+y,t+z+[x,y]/2).
	\]
	In orthonormal coordinates, the horizontal fields and the sub-Laplacian are
	\begin{equation}\label{eq1.2}
		X_i=\partial_{x_i}+\frac12\sum_{\alpha=1}^n(J_\alpha x)_i\partial_{t_\alpha},
		\qquad \nabla_G=(X_1,\ldots,X_m),\qquad
		\Delta_G=\sum_{i=1}^mX_i^2.
	\end{equation}
	Haar measure is $\dd g=\dd x\dd t$, and the dilations are
	$\delta_\lambda(x,t)=(\lambda x,\lambda^2t)$. In particular, $m\geq2$, $n\geq1$ and $Q\geq4$. All norms on $V$ and $Z$ refer to the fixed H-type metric. Positive constants denoted by $C$ or $C_G$ may change from line to line; $C_G$ depends only on the fixed group.
	Let $\dS(G)$ denote the completion of $C_c^\infty(G)$ for the norm
	$\norm u_{\dS}=\norm{\nabla_Gu}_2$, and set
	\[
	\cE(u,v)=\int_G\ip{\nabla_Gu}{\nabla_Gv}\dd g,
	\qquad \cE(u)=\cE(u,u).
	\]
	The Folland-Stein embedding realizes this completion as a subspace of $L^{2^*}(G)$ \cite{FischerRuzhansky2016,Folland1975,FollandStein1982}. We work with real-valued functions in this homogeneous energy space. Denoting the best constant by $S_G$, we write the critical inequality as
	\begin{equation}\label{eq1.3}
		\cE(u)\geq S_G\norm u_{2^*}^2,
		\qquad u\in\dS(G).
	\end{equation}
	
	In Euclidean space, Aubin and Talenti \cite{Aubin1976,Talenti1976} determined the sharp Sobolev constant on $\R^d$, $d\geq3$, and classified the extremals. With
	\[
	W_{\lambda,y}(x)=\lambda^{(d-2)/2}
	(1+\lambda^2|x-y|^2)^{-(d-2)/2},
	\qquad
	\mathfrak M_d=\{cW_{\lambda,y}:c\in\R,\ \lambda>0,\ y\in\R^d\},
	\]
	the equality set is $\mathfrak M_d$. Answering a question of Brezis and Lieb \cite{BrezisLieb1985}, Bianchi and Egnell \cite{BianchiEgnell1991} proved that
	\[
	\norm{\nabla f}_2^2-S_d\norm f_{2d/(d-2)}^2
	\geq C_d\inf_{z\in\mathfrak M_d}\norm{\nabla(f-z)}_2^2,
	\qquad f\in\dot H^1(\R^d),
	\]
	for some $C_d>0$. Their proof combines a spectral estimate near $\mathfrak M_d$ with compactness of minimizing sequences modulo the symmetries. Chen, Frank and Weth \cite{ChenFrankWeth2013} proved the corresponding fractional Sobolev remainder. Dolbeault, Esteban, Figalli, Frank and Loss \cite{DolbeaultEstebanFigalliFrankLoss2025} obtained constructive stability estimates with the optimal order of dependence on the dimension. We establish the spectral classification and compactness needed for this approach on an arbitrary H-type group.
	
	The sharp form of \eqref{eq1.3} is closely related to the critical equation for the sub-Laplacian. On the Heisenberg group, Jerison and Lee \cite{JerisonLee1988} determined all Sobolev extremals in their study of the CR Yamabe problem. Frank and Lieb \cite{FrankLieb2012} subsequently proved the sharp~Hardy-Littlewood-Sobolev inequalities and their dual Sobolev inequalities, including conformally invariant fractional orders. Sharp inequalities and extremal classifications in the quaternionic and octonionic cases were obtained by Ivanov, Minchev and Vassilev \cite{IvanovMinchevVassilev2012} and Christ, Liu and Zhang \cite{ChristLiuZhang2016Octonionic}. These cases possess additional conformal geometry associated with their Iwasawa structure.
	
	For a general H-type group, put
	\[
	\begin{gathered}
		s=\frac{|x|^2}{4},\qquad D=(1+s)^2+|t|^2,
		\qquad \rho^4=s^2+|t|^2,\qquad U=D^{-(Q-2)/4},\\
		\kappa=\frac{m(Q-2)}4,\qquad
		\mu_\star=\frac{m(Q+2)}4=(2^*-1)\kappa.
	\end{gathered}
	\]
	The explicit solutions studied by Garofalo and Vassilev \cite{GarofaloVassilev2001} satisfy
	\begin{equation}\label{eq1.4}
		-\Delta_GU=\kappa U^{2^*-1}.
	\end{equation}
	Group translations and dilations act isometrically on both $\dS(G)$ and $L^{2^*}(G)$:
	\begin{equation}\label{eq1.5}
		(\cU_{\lambda,\eta}u)(g)
		=\lambda^{(Q-2)/2}u\bigl(\delta_\lambda(\eta^{-1}g)\bigr),
		\qquad \lambda>0,\quad \eta\in G.
	\end{equation}
	We distinguish the normalized solution family from the cone with arbitrary amplitude:
	\begin{equation}\label{eq1.6}
		\begin{gathered}
			U_{\lambda,\eta}=\cU_{\lambda,\eta}U,\qquad
			\cM=\{U_{\lambda,\eta}:\lambda>0,\ \eta\in G\},\\
			\mathfrak M=\{cU_{\lambda,\eta}:c\in\R,\ \lambda>0,\ \eta\in G\}.
		\end{gathered}
	\end{equation}
	Q.~Yang \cite[Theorem~1.2]{Yang2024} proved that $U$ attains the sharp constant in \eqref{eq1.3}. In the normalization above, this gives
	\[
	S_G=\frac{\cE(U)}{\norm U_{2^*}^2}
	=\kappa\left(\int_G U^{2^*}\dd g\right)^{2/Q}.
	\]
	Yang's proof combines subcritical approximation with the second-variation arguments of Frank-Lieb and Hang-Wang \cite{FrankLieb2012,HangWang2022}. The present work starts from this sharp inequality and addresses its full equality set and stability. Wang and Yang \cite{WangYang2024} obtained sharp inequalities for conformally invariant fractional powers on H-type groups.
	
	The conjecture of Garofalo and Vassilev concerns all positive entire solutions of finite energy for the critical equation. Related symmetry and nonlinear Liouville results appear in \cite{BonfiglioliUguzzoni2004,GarofaloVassilev2001}. A recent preprint of Lin \cite{Lin2026} establishes this classification on the octonionic Heisenberg group of dimension fifteen. Our classification concerns the extremals of \eqref{eq1.3} and uses their minimizing property. Define
	\[
	\Def(u)=\cE(u)-S_G\norm u_{2^*}^2,
	\qquad d(u)=\inf_{z\in\mathfrak M}\norm{u-z}_{\dS}.
	\]
	Our first main result identifies the entire equality set and controls the distance to it.
	
	\begin{theorem}\label{Thm1.1}
		The equality set in \eqref{eq1.3} is exactly $\mathfrak M$. Moreover, there exists $c_G>0$ such that
		\begin{equation}\label{eq1.7}
			\Def(u)\geq c_Gd(u)^2
			\qquad\text{for every real }u\in\dS(G).
		\end{equation}
	\end{theorem}
	
	Theorem~\ref{Thm1.1} holds for each fixed H-type group. The constant $c_G$ may depend on the group. The distance is taken to the full cone \eqref{eq1.6}, since every real multiple of a bubble has zero deficit.
	
	On the Heisenberg group, Loiudice \cite{Loiudice2005} established Sobolev remainder estimates. Liu and Zhang \cite{LiuZhang2015} proved stability estimates for the conformally invariant fractional Sobolev and Hardy-Littlewood-Sobolev inequalities and stated analogous results for the other Iwasawa groups. Chen, Lu, Tang and Wang \cite{ChenLuTangWang2026} obtained dimension-dependent bounds of asymptotically sharp order by combining a local estimate with the CR Yamabe flow. Tang, Zhang and Zhang \cite{TangZhangZhang2024} studied attainment of the optimal Heisenberg stability quotient. We prove Theorem~\ref{Thm1.1} directly on the group, using the Clifford relations and the variational properties of extremals.
	
	The second main result supplies the necessary spectral information. Consider the weighted eigenvalue problem
	\begin{equation}\label{eq1.8}
		-\Delta_Gv=\mu D^{-1}v,
		\qquad v\in\dS(G).
	\end{equation}
	The embedding into $L^2(G,D^{-1}\dd g)$ is compact. Yang \cite[Theorem~1.3]{Yang2024} proved that the first eigenvalue is $\kappa$, with a one-dimensional eigenspace, and that the second eigenvalue is $\mu_\star$. Translation and dilation derivatives of the bubble are eigenfunctions at $\mu_\star$. Remark~1.1 of \cite{Yang2024} asks whether the second eigenspace is exactly the $(m+n+1)$-dimensional span of these derivatives, as in the Euclidean result of Bianchi and Egnell \cite[Lemma~A1]{BianchiEgnell1991}.
	
	On the Heisenberg group, nondegeneracy is used in the perturbation theory of Malchiodi and Uguzzoni \cite{MalchiodiUguzzoni2002} and in the construction of concentrating sign-changing solutions by Qiang, Tang and Zhang \cite{QiangTangZhang2026}. The following theorem identifies the linearized kernel on every H-type group.
	
	\Needspace{12\baselineskip}
	\begin{theorem}\label{Thm1.2}
		Every distributional solution of
		\begin{equation}\label{eq1.9}
			-\Delta_Gv=\mu_\star D^{-1}v,
			\qquad v\in\dS(G),
		\end{equation}
		has the form
		\begin{equation}\label{eq1.10}
			v=D^{-(Q+2)/4}
			\left[c_0(1-\rho^4)+2\ip{\tau}{t}+\ip{c}{(1+s)x-J_tx}\right],
		\end{equation}
		where $c_0\in\R$, $c\in V$ and $\tau\in Z$. Conversely, all these functions are finite-energy solutions. Hence the linearized kernel consists exactly of the translation and dilation derivatives and has dimension $m+n+1$.
	\end{theorem}
	
	The sign in \eqref{eq1.10} follows from the convention \eqref{eq1.2}. The isometries \eqref{eq1.5} transport the classification to every $U_{\lambda,\eta}$. Applying the theorem to the real and imaginary parts gives the corresponding complex-valued classification. Corollary~\ref{Cor4.8} states the equivalent finite-energy Liouville theorem.
	
	For $\cZ=\ker_{\dS}(-\Delta_G-\mu_\star D^{-1})$, the tangent space to the extremal cone is
	\[
	T_U\mathfrak M=\Span\{U\}\oplus\cZ,
	\qquad \dim T_U\mathfrak M=m+n+2.
	\]
	The second variation of the Sobolev deficit is coercive on the orthogonal complement of this full tangent space. Let $\mu_{\mathrm{next}}$ denote the next distinct eigenvalue of \eqref{eq1.8} after $\mu_\star$. Proposition~\ref{Prop5.4} gives the explicit bounds
	\begin{equation}\label{eq1.11}
		\mu_\star+\frac{2\mu_\star}{2m+3n+7}
		\leq\mu_{\mathrm{next}}\leq\mu_\star+m.
	\end{equation}
	Consequently, the coefficient in the transverse coercivity estimate is
	\[
	c_{\mathrm{lin}}=1-\frac{\mu_\star}{\mu_{\mathrm{next}}},
	\qquad
	\frac{2}{2m+3n+9}\leq c_{\mathrm{lin}}\leq\frac4{Q+6}.
	\]
	Proposition~\ref{Prop5.6} identifies this coefficient by the exact asymptotic formula
	\[
	\lim_{\varepsilon\downarrow0}
	\inf_{0<d(u)\leq\varepsilon\norm u_{\dS}}
	\frac{\Def(u)}{d(u)^2}=c_{\mathrm{lin}}.
	\]
	More precisely, for $d(u)/\norm u_{\dS}$ sufficiently small,
	\[
	\Def(u)\geq
	\left[c_{\mathrm{lin}}-C_G
	\left(\frac{d(u)}{\norm u_{\dS}}\right)^\sigma\right]d(u)^2,
	\qquad \sigma=\min\left\{1,\frac4{Q-2}\right\}.
	\]
	Corollary~\ref{Cor5.7} characterizes sequences attaining this asymptotic coefficient. After modulation, the normal remainders, normalized to have unit energy norm, approach the $\mu_{\mathrm{next}}$ eigenspace. The coefficient $c_{\mathrm{lin}}$ describes the local asymptotics. The optimal global constant in \eqref{eq1.7} is a separate variational quantity. In Euclidean space, K\"onig \cite{Konig2023} proved that the optimal global Bianchi-Egnell constant is strictly smaller than the local spectral coefficient $4/(d+4)$. He also proved attainment of the optimal global constant \cite{Konig2025}.
	
	There is also a consequence for the critical equation. For
	$\cF(u)=-\Delta_Gu-\kappa|u|^{4/(Q-2)}u$, Theorem~\ref{Thm5.8} proves
	\[
	C_G^{-1}\dist_{\dS}(u,\cM)
	\leq\norm{\cF(u)}_{(\dS)^*}
	\leq C_G\dist_{\dS}(u,\cM)
	\]
	in a fixed neighborhood of the normalized family $\cM$. Thus the residual norm and the distance to a single normalized bubble are locally equivalent. For several weakly interacting Heisenberg bubbles, quantitative global compactness was studied by Chen, Fan and Liao \cite{ChenFanLiao2025}.
	
	The proof has a variational part and a spectral part. For the variational argument, we center a probability measure $\nu$ on $G$ by minimizing
	\[
	\Phi_\nu(a,\eta)
	=\int_G\log\frac{D_a(\eta^{-1}g)}{aD(g)}\dd\nu(g),
	\qquad
	D_a(z,\tau)=\left(a+\frac{|z|^2}{4}\right)^2+|\tau|^2.
	\]
	The functional is finite for every probability measure. Its sublevel sets are compact exactly when every atom of $\nu$ has mass less than $1/2$. Differentiation at a minimum gives zero moments against the logarithmic derivatives of $U$. After normalizing $|u|^{2^*}$ to have unit mass, these moment conditions make the symmetry directions admissible tests in the second variation at any extremal. Yang's summed variation identity gives
	\[
	\cE(u)\leq\kappa\int_GD^{-1}u^2\dd g
	\]
	for the centered extremal. The ground-state identity gives the reverse inequality and shows that equality holds exactly for $u=cU$. The extremal classification therefore precedes the spectral argument. The same moment conditions exclude concentration of a minimizing sequence at a single point of $G$ or at infinity. Energy splitting and the Brezis-Lieb lemma \cite{BrezisLieb1983} then give compactness modulo translations and dilations.
	
	The spectral part starts from the formula
	\[
	\Delta_G=\Delta_x+s\Delta_t+
	\sum_{\alpha=1}^n\Omega_\alpha\partial_{t_\alpha},
	\qquad \Omega_\alpha=(J_\alpha x)\cdot\nabla_x.
	\]
	Each $\Omega_\alpha$ preserves Euclidean spherical harmonic degree in $x$. This gives an orthogonal decomposition of the energy space by horizontal harmonic degree. An identity expressing the tensor energy as a sum of squares, combined with a scalar ground-state transform, excludes all horizontal components of degree at least two at the eigenvalue $\mu_\star$. A quantitative refinement in degree two gives the lower bound in \eqref{eq1.11}. In degrees zero and one, a change of variables followed by scalar or matrix conjugation reduces the problem to weighted operators on an auxiliary Euclidean ball. A Clifford harmonic decomposition diagonalizes the angular operator containing the products $J_\alpha J_\beta$. The polynomial and Clifford decompositions are classical \cite{BrackxDelangheSommen1982,DunklXu2014}. We prove the conjugation formulas, the estimates and their extensions to the completed energy spaces. The local quadratic estimate follows from the spectral analysis and parameter modulation. Compactness then gives the global estimate by the Bianchi-Egnell argument.
	
	Section~\ref{Sec2} proves centering, classification of extremals and compactness of minimizing sequences. Section~\ref{Sec3} establishes the horizontal harmonic reduction. Section~\ref{Sec4} treats the two lowest degrees and proves Theorem~\ref{Thm1.2}. Section~\ref{Sec5} gives the spectral bounds, the local asymptotics, the proof of Theorem~\ref{Thm1.1}, and the residual estimate.
	
	\section{Centering, extremals and compactness}\label{Sec2}
	
	The variational and spectral arguments both use the weighted form
	\[
	\cN(u,v)=\int_GD^{-1}uv\dd g,\qquad \cN(u)=\cN(u,u).
	\]
	
	\begin{lemma}\label{Lem2.1}
		The form $\cN$ is continuous on $\dS(G)$, and the embedding
		\[
		\dS(G)\longrightarrow L^2(G,D^{-1}\dd g)
		\]
		is compact. Moreover,
		\begin{equation}\label{eq2.1}
			\cE(u)-\kappa\cN(u)
			=\int_GU^2\abs{\nabla_G(u/U)}^2\dd g
			\quad (u\in\dS(G)).
		\end{equation}
		Equality in $\cE(u)\geq\kappa\cN(u)$ holds precisely when $u$ is a constant multiple of $U$.
	\end{lemma}
	\begin{proof}
		From \eqref{eq1.1} and \eqref{eq1.2},
		\begin{equation}\label{eq2.2}
			\abs{\nabla_GD}^2=\abs{x}^2D,
			\qquad
			\Delta_GD=m+\frac{Q+2}{4}\abs{x}^2.
		\end{equation}
		These identities give \eqref{eq1.4}. They also imply $U=O(\rho^{2-Q})$ and $\nabla_GU=O(\rho^{1-Q})$ at infinity. Thus $U$ has finite energy. A homogeneous cutoff with $\abs{\nabla_G\chi_R}\leq C/R$ on $B_{2R}\setminus B_R$ gives
		\[
		\int_{B_{2R}\setminus B_R}U^2\abs{\nabla_G\chi_R}^2\dd g
		\leq C R^{2-Q}\longrightarrow0,
		\]
		so $U\in\dS(G)$. For $u\in C_c^\infty(G)$, expansion of the square and integration by parts give
		\[
		\int_GU^2\abs{\nabla_G(u/U)}^2\dd g
		=\cE(u)+\int_G\frac{\Delta_GU}{U}u^2\dd g
		=\cE(u)-\kappa\cN(u).
		\]
		Hence $\cN(u)\leq\kappa^{-1}\cE(u)$. Both the weighted mass and the quotient gradient extend by completion. On a compact subset, $D^{-1}$ is bounded below and $U$ is smooth and bounded away from zero. Weighted $L^2$ convergence and convergence of the horizontal gradients therefore identify the limits with $u/U$ and its distributional horizontal derivatives. Thus \eqref{eq2.1} holds on $\dS(G)$. If its right-hand side vanishes, $X_i(u/U)=0$ for every $i$. Their commutators span the central directions. All distributional Euclidean derivatives of $u/U$ therefore vanish, and $u/U$ is constant.
		
		For compactness, let $B_R=\{\rho<R\}$. Homogeneous polar coordinates give
		\[
		\int_{G\setminus B_R}D^{-Q/2}\dd g\leq C R^{-Q},\qquad R\geq1.
		\]
		H\"older's inequality and the nonsharp Sobolev inequality imply
		\begin{equation}\label{eq2.3}
			\int_{G\setminus B_R}D^{-1}\abs{u}^2\dd g
			\leq
			\norm{u}_{2^*}^2
			\left(\int_{G\setminus B_R}D^{-Q/2}\dd g\right)^{2/Q}
			\leq C R^{-2}\cE(u).
		\end{equation}
		For $\chi\in C_c^\infty(G)$, the product estimate
		\[
		\norm{\nabla_G(\chi u)}_2
		\leq\norm\chi_\infty\norm{\nabla_Gu}_2
		+\norm{\nabla_G\chi}_Q\norm u_{2^*}
		\]
		shows that localization is continuous in energy. The Sobolev and H\"older inequalities also give local $L^2$ bounds. The localized sequence is therefore bounded in the inhomogeneous Sobolev space of order one associated with $-\Delta_G$. The local embedding in \cite[Theorem~4.4.24]{FischerRuzhansky2016}, applied with $p=2$ and $s=1$, bounds it in $H^{1/2}_{\mathrm{loc}}(\R^{m+n})$, since the largest dilation weight is two. The Euclidean Rellich theorem gives a subsequence converging in $L^2$ on each fixed ball. Estimate~\eqref{eq2.3} controls the weighted tails and proves compactness.
	\end{proof}
	
	Define the moment functions
	\begin{equation}\label{eq2.4}
		\omega(x,t)=\left(\frac{(1+s)x-J_tx}{D},\ \frac{2t}{D},\ \frac{1-\rho^4}{D}\right).
	\end{equation}
	These are the normalized logarithmic derivatives of the bubble family at $(1,e)$. Left translation of a function has infinitesimal generator equal to the negative of the corresponding right-invariant field.
	
	\begin{lemma}\label{Lem2.2}
		The moment functions satisfy
		\begin{equation}\label{eq2.5}
			\sum_{i=1}^{m+n+1}\omega_i^2=1,\qquad
			\sum_{i=1}^{m+n+1}|\nabla_G\omega_i|^2=mD^{-1}.
		\end{equation}
		For every $u\in\dS(G)$, each $u\omega_i$ belongs to $\dS(G)$ and
		\begin{equation}\label{eq2.6}
			\sum_{i=1}^{m+n+1}\cE(u\omega_i)=\cE(u)+m\cN(u).
		\end{equation}
	\end{lemma}
	\begin{proof}
		Clifford anticommutation gives $|((1+s)\Id-J_t)x|^2=D|x|^2$. The identity
		\[
		4sD+4|t|^2+(1-s^2-|t|^2)^2=D^2
		\]
		then proves the first formula in \eqref{eq2.5}. Differentiation of \eqref{eq1.4} along the bubble parameters gives
		\[
		-\Delta_G(U\omega_i)=\mu_\star D^{-1}U\omega_i.
		\]
		Subtract the equation for $U$ and divide by $U$ to obtain
		\[
		\Delta_G\omega_i+2\ip{\nabla_G\log U}{\nabla_G\omega_i}
		=-mD^{-1}\omega_i.
		\]
		Multiply by $\omega_i$ and sum. Since $\sum_i\omega_i\nabla_G\omega_i=0$, the second identity in \eqref{eq2.5} follows. Expanding $\nabla_G(u\omega_i)$ gives \eqref{eq2.6} on $C_c^\infty(G)$. Lemma~\ref{Lem2.1} then shows that multiplication by each $\omega_i$ is bounded on $\dS(G)$, so the identity extends by density. These formulas agree with \cite[Lemmas~2.2 and~3.4]{Yang2024} under the convention \eqref{eq1.2}.
	\end{proof}
	
	For $a>0$ and $h=(z,\tau)\in G$, write
	\[
	D_a(h)=\left(a+\frac{|z|^2}{4}\right)^2+|\tau|^2.
	\]
	The relation with bubble parameters is
	\begin{equation}\label{eq2.7}
		U_{a^{-1/2},\eta}(g)
		=\left(\frac{a}{D_a(\eta^{-1}g)}\right)^{(Q-2)/4}.
	\end{equation}
	The reference denominator $D(g)$ gives a bounded logarithmic ratio for each fixed pair of parameters.
	
	\begin{lemma}\label{Lem2.3}
		Put
		\[
		A=1+a+\rho(\eta)^2,\qquad T=\frac{A^2}{a},\qquad
		R_{a,\eta}(g)=\frac{D_a(\eta^{-1}g)}{aD(g)}.
		\]
		There is $C\geq1$, depending only on $G$, such that
		\begin{equation}\label{eq2.8}
			C^{-1}T^{-1}\leq R_{a,\eta}(g)\leq CT
			\qquad(a>0,\ \eta,g\in G).
		\end{equation}
		Moreover, $T\geq4$ and its sublevel sets in $(0,\infty)\times G$ are compact.
	\end{lemma}
	\begin{proof}
		The elementary bounds
		\[
		\frac12(a+\rho(h)^2)^2\leq D_a(h)\leq(a+\rho(h)^2)^2
		\]
		follow by expanding the square and using $|z|^2/4\leq\rho(h)^2$. The analogous bounds hold for $D(g)$ with $a=1$. A homogeneous quasi-norm satisfies
		\[
		\rho(gh)\leq C_0\bigl(\rho(g)+\rho(h)\bigr),\qquad
		\rho(g^{-1})=\rho(g).
		\]
		For $h=\eta^{-1}g$, these inequalities imply
		\[
		a+\rho(h)^2\leq C_1A(1+\rho(g)^2),
		\qquad
		1+\rho(g)^2\leq C_1\frac{A}{a}(a+\rho(h)^2).
		\]
		Substitution gives \eqref{eq2.8}. Finally,
		$T\geq(a+1)^2/a\geq4$. On each sublevel set of $T$, the parameter $a$ is bounded above and bounded away from zero, and $\rho(\eta)$ is bounded. The sublevel set is therefore contained in a compact subset of $(0,\infty)\times G$. It is closed because $T$ is continuous, and hence is compact.
	\end{proof}
	
	\begin{lemma}\label{Lem2.4}
		Let $\nu$ be a Borel probability measure on $G$. The functional
		\[
		\Phi_\nu(a,\eta)=\int_G\log R_{a,\eta}(g)\dd\nu(g)
		\]
		is finite and continuously differentiable. Its sublevel sets are compact if and only if
		\begin{equation}\label{eq2.9}
			\nu(\{\xi\})<\frac12\qquad\text{for every }\xi\in G.
		\end{equation}
		Under this condition it attains a minimum at an interior point of $(0,\infty)\times G$. In particular, the conclusion holds for every atomless probability measure.
	\end{lemma}
	\begin{proof}
		Lemma~\ref{Lem2.3} gives $|\log R_{a,\eta}(g)|\leq\log T+C$ uniformly in $g$. The dominated convergence theorem proves finiteness and local continuity of $\Phi_\nu$. Its parameter derivatives are also locally uniformly bounded in $g$. For example,
		\[
		|\partial_a\log D_a|\leq2/a,\qquad
		|\partial_a\log R_{a,\eta}|\leq3/a,
		\]
		and a central translation derivative has absolute value at most $1/a$. Writing $z=x-y$ and $\tau=t-w-[y,x]/2$ for $\eta=(y,w)$, a horizontal translation derivative has absolute value at most
		\[
		C\bigl(a^{-1/2}+|y|a^{-1}\bigr).
		\]
		This follows from $|[y,x]|\leq C|y||x|$, $|x|\leq|z|+|y|$, and
		$|\tau|/((a+|z|^2/4)^2+|\tau|^2)\leq1/(2(a+|z|^2/4))$.
		Dominated differentiation now proves the $C^1$ assertion.
		
		Assume \eqref{eq2.9}. Let $T_j\to\infty$. After passage to a subsequence, either $A_j\to\infty$ or $A_j$ is bounded.
		
		If $A_j\to\infty$, choose a compact set $K\subset G$ with $\nu(K)>3/4$. On $K$,
		\[
		a_j+\rho(\eta_j^{-1}g)^2\geq c_KA_j
		\]
		for all sufficiently large $j$. To see this, either $a_j\geq A_j/3$, or $\rho(\eta_j)^2$ is comparable to $A_j$ and the reverse quasi-triangle bound gives
		$\rho(\eta_j^{-1}g)\geq c\rho(\eta_j)$ uniformly on $K$.
		Since $D$ is bounded on $K$, we have
		\begin{equation}\label{eq2.10}
			R_{a_j,\eta_j}(g)\geq c_KT_j\qquad(g\in K).
		\end{equation}
		
		If $A_j$ is bounded, then $a_j\to0$ and, after a further subsequence, $\eta_j\to\eta_\infty\in G$. Since $\nu(\{\eta_\infty\})<1/2$, inner regularity gives a compact set $K\subset G\setminus\{\eta_\infty\}$ with $\nu(K)>1/2$. The distance from $\eta_j$ to $K$ is bounded below for large $j$. Thus \eqref{eq2.10} holds in this case too, because $A_j$ is bounded.
		
		In either case, integrate \eqref{eq2.10} over $K$ and the lower bound in \eqref{eq2.8} over its complement. We obtain
		\[
		\Phi_\nu(a_j,\eta_j)
		\geq(2\nu(K)-1)\log T_j-C_K\longrightarrow\infty.
		\]
		The properness of $T$ now implies compactness of every sublevel set of $\Phi_\nu$. A minimizing sequence in the nonempty compact sublevel $\Phi_\nu\leq\Phi_\nu(1,e)+1$ therefore has a subsequence converging to an interior minimum.
		
		Conversely, suppose $\nu(\{\xi\})=\vartheta\geq1/2$. Fix $\eta=\xi$ and let $a\downarrow0$. At the atom,
		\[
		R_{a,\xi}(\xi)=a/D(\xi).
		\]
		For $0<a\leq1$, the upper bound in \eqref{eq2.8} gives $R_{a,\xi}(g)\leq C_\xi/a$ for every $g$. Splitting off the atom yields
		\[
		\Phi_\nu(a,\xi)
		\leq(2\vartheta-1)\log a+C_\xi.
		\]
		For $\vartheta>1/2$, the right-hand side tends to $-\infty$. For $\vartheta=1/2$, it is uniformly bounded above. In either case, a fixed sublevel set contains the escaping sequence $(a,\xi)$ with $a\downarrow0$ and is noncompact.
	\end{proof}
	
	\begin{proposition}\label{Prop2.5}
		For every nonzero real $u\in L^{2^*}(G)$, there are $\lambda>0$ and $\eta\in G$ such that
		\[
		v=\cU_{\lambda,\eta}^{-1}u
		\quad\hbox{satisfies}\quad
		\int_G|v|^{2^*}\omega_i\dd g=0\quad(1\leq i\leq m+n+1).
		\]
		If $u\in\dS(G)$, this transformation also preserves its energy.
	\end{proposition}
	\begin{proof}
		Apply Lemma~\ref{Lem2.4} to
		$\dd\nu=|u|^{2^*}\dd g/\norm{u}_{2^*}^{2^*}$, which is atomless. Let $(a_0,\eta_0)$ be a minimum and put $\lambda_0=a_0^{-1/2}$. By \eqref{eq2.7}, the integral of every parameter derivative of $\log U_{\lambda,\eta}$ vanishes there. We transfer this stationarity condition to the reference bubble $U=U_{1,e}$.
		
		Composition of the similarities gives
		\begin{equation}\label{eq2.11}
			\cU_{\lambda_0,\eta_0}\cU_{\lambda',\eta'}
			=\cU_{\lambda_0\lambda',\,\eta_0\delta_{\lambda_0^{-1}}\eta'}.
		\end{equation}
		The right-hand parameters in \eqref{eq2.11} form a local coordinate system at $(\lambda_0,\eta_0)$. Write $h=\delta_{\lambda_0}(\eta_0^{-1}g)$. Then
		\[
		\log(\cU_{\lambda_0,\eta_0}U_{\lambda',\eta'})(g)
		=\frac{Q-2}{2}\log\lambda_0+\log U_{\lambda',\eta'}(h).
		\]
		Moreover
		\[
		v(h)=\lambda_0^{-(Q-2)/2}u(\eta_0\delta_{\lambda_0^{-1}}h),
		\qquad |v(h)|^{2^*}\dd h=|u(g)|^{2^*}\dd g.
		\]
		At $(\lambda',\eta')=(1,e)$, the horizontal and central derivatives of $\log U_{\lambda',\eta'}$ are $(Q-2)/4$ times the corresponding components of $\omega$. The dilation derivative is $(Q-2)/2$ times its last component. Stationarity and the change of variables prove all the zero moment conditions.
	\end{proof}
	
	\begin{theorem}\label{Thm2.6}
		For a real $u\in\dS(G)$, equality in
		$\cE(u)\geq S_G\norm{u}_{2^*}^2$ holds if and only if $u\in\mathfrak M$.
	\end{theorem}
	\begin{proof}
		By \cite[Theorem~1.2]{Yang2024} and similarity invariance, every member of $\mathfrak M$ attains equality. Conversely, let $u\neq0$ be an extremal and apply Proposition~\ref{Prop2.5}. Denote the centered extremal by $v$ and set $B=\int_G|v|^{2^*}\dd g$.
		
		The critical embedding and $2^*>2$ make the quotient $\cE(v)/\norm{v}_{2^*}^2$ twice continuously differentiable on $\dS(G)\setminus\{0\}$. In particular,
		\[
		\int_G|v|^{2^*-2}f^2\dd g
		\leq\norm v_{2^*}^{2^*-2}\norm f_{2^*}^2,
		\]
		so the second variation is defined for every energy direction at each real $v\neq0$. At an extremal it gives
		\[
		\cE(f)\geq(2^*-1)\frac{\cE(v)}B\int_G|v|^{2^*-2}f^2\dd g
		\quad\hbox{when}\quad
		\int_G|v|^{2^*-2}vf\dd g=0.
		\]
		Take $f=v\omega_i$. Membership in the energy space follows from Lemma~\ref{Lem2.2}, and the tangent condition is precisely $\int_G|v|^{2^*}\omega_i\dd g=0$. Summing the inequalities and using $|\omega|^2=1$ gives
		\[
		\cE(v)+m\cN(v)\geq(2^*-1)\cE(v).
		\]
		Since $\kappa=m/(2^*-2)$, this says $\cE(v)\leq\kappa\cN(v)$. The opposite inequality and its equality characterization are given by Lemma~\ref{Lem2.1}. Hence $v=cU$. Undoing the similarity gives $u=cU_{\lambda,\eta}$.
	\end{proof}
	
	The classification uses the nonnegativity of the constrained second variation at a Sobolev minimizer. It applies to real extremals of either sign and is independent of the kernel classification in Theorem~\ref{Thm1.2}.
	
	The map \eqref{eq2.4} extends continuously to the one-point compactification $\widehat G=G\cup\{\infty\}$, with
	\begin{equation}\label{eq2.12}
		\omega(\infty)=(0,0,-1),\qquad |\omega(\xi)|=1\quad(\xi\in\widehat G).
	\end{equation}
	This follows from $D\sim\rho^4$ at infinity, the identity $|((1+s)\Id-J_t)x|=|x|\sqrt D$, and the bounds $|x|\leq2\rho$ and $|t|\leq\rho^2$.
	
	\begin{theorem}\label{Thm2.7}
		Suppose $v_j\in\dS(G)$ is real and
		\[
		\int_G|v_j|^{2^*}\dd g=1,\qquad \cE(v_j)\to S_G,\qquad
		\int_G|v_j|^{2^*}\omega\dd g=0.
		\]
		A subsequence converges strongly in $\dS(G)$ to $cU$, where
		$|c|=(\int_GU^{2^*}\dd g)^{-1/2^*}$. Consequently every normalized minimizing sequence is precompact modulo group translations and dilations.
	\end{theorem}
	\begin{proof}
		Pass to a subsequence such that $v_j\rightharpoonup v$ in $\dS(G)$, $v_j\to v$ in $L^2_{\mathrm{loc}}(G)$, and $v_j\to v$ almost everywhere. Put $A=\int_G|v|^{2^*}\dd g$. Hilbert-space splitting and the Brezis-Lieb lemma \cite{BrezisLieb1983} give
		\[
		\cE(v_j)=\cE(v)+\cE(v_j-v)+o(1),\qquad
		\int_G|v_j-v|^{2^*}\dd g\longrightarrow1-A.
		\]
		Applying the sharp inequality to the two terms yields
		\[
		1\geq A^{2/2^*}+(1-A)^{2/2^*}.
		\]
		Since $2/2^*<1$, strict concavity gives $A\in\{0,1\}$. If $A=1$, the sharp inequality gives $\cE(v)\geq S_G$, while weak lower semicontinuity gives $\cE(v)\leq S_G$. The energy splitting then implies $\cE(v_j-v)\to0$. It remains to exclude $v=0$.
		
		Assume $v=0$. On the compact space $\widehat G$, pass to weak limits of finite measures
		\[
		|v_j|^{2^*}\dd g\rightharpoonup\nu,\qquad
		|\nabla_Gv_j|^2\dd g\rightharpoonup\mu.
		\]
		Their total masses are $\nu(\widehat G)=1$ and $\mu(\widehat G)=S_G$. For $\varphi\in C_c^\infty(G)$, local $L^2$ convergence gives
		\[
		\cE(\varphi v_j)=\int_G\varphi^2|\nabla_Gv_j|^2\dd g+o(1).
		\]
		Thus
		\begin{equation}\label{eq2.13}
			S_G\left(\int_G|\varphi|^{2^*}\dd\nu\right)^{2/2^*}
			\leq\int_G\varphi^2\dd\mu.
		\end{equation}
		On a fixed compact set, regularity of the finite measure $\nu+\mu$ gives smooth compactly supported approximations of Borel indicators in both $L^{2^*}(\nu)$ and $L^2(\mu)$. Thus \eqref{eq2.13} holds for these indicators. Suppose the nonatomic part of $\nu$ has positive mass on a compact set $K$. Remove the at most countably many atoms and partition the remaining set into Borel sets of $\nu$-mass at most $\varepsilon$. Summing the indicator inequalities gives
		\[
		\mu(\widehat G)\geq
		S_G\varepsilon^{2/2^*-1}\nu_{\mathrm{na}}(K).
		\]
		Since $2/2^*-1<0$, letting $\varepsilon\to0$ contradicts the finiteness of $\mu$. Hence
		\[
		\nu|_G=\sum_i\alpha_i\delta_{\xi_i},\qquad
		\mu(\{\xi_i\})\geq S_G\alpha_i^{2/2^*}.
		\]
		To prove the corresponding bound at infinity, choose a smooth cutoff $\chi_R$ equal to zero on $\{\rho<R\}$ and to one on $\{\rho>2R\}$. For fixed $R$, the errors in $\cE(\chi_Rv_j)$ tend to zero by local $L^2$ convergence on the annulus and the uniform energy bound. The mixed term, for example, satisfies
		\[
		2\norm{\nabla_Gv_j}_2
		\left(\int_{B_{2R}\setminus B_R}|v_j|^2|\nabla_G\chi_R|^2\dd g\right)^{1/2}
		\longrightarrow0.
		\]
		The function $\chi_Rv_j$ belongs to $\dS(G)$, by multiplying a compactly supported approximation of $v_j$ by $\chi_R$. Passing first $j\to\infty$ and then $R\to\infty$ gives
		\[
		\mu(\{\infty\})\geq S_G\alpha_\infty^{2/2^*},
		\qquad \alpha_\infty=\nu(\{\infty\}).
		\]
		It follows that
		\[
		1\geq\sum_i\alpha_i^{2/2^*}+\alpha_\infty^{2/2^*},
		\qquad \sum_i\alpha_i+\alpha_\infty=1.
		\]
		The strict inequality $r<r^{2/2^*}$ for $0<r<1$ forces exactly one mass to equal one. Thus $\nu=\delta_\xi$ for some $\xi\in\widehat G$. The zero moment condition and \eqref{eq2.12} then give
		\[
		0=\lim_j\int_G\omega|v_j|^{2^*}\dd g=\omega(\xi),
		\]
		contradicting $|\omega(\xi)|=1$.
		
		Hence $A=1$, and $v_j\to v$ strongly in $\dS(G)$ and $L^{2^*}(G)$. Boundedness of $\omega$ allows passage to the limit in the zero moments. The proof of Theorem~\ref{Thm2.6} therefore gives $v=cU$. The last assertion follows by applying Proposition~\ref{Prop2.5} to each member of a normalized minimizing sequence.
	\end{proof}
	
	\section{Horizontal harmonic reduction}\label{Sec3}
	Throughout the remaining sections, set
	\[
	b=\frac m2,\qquad \HH=(0,\infty)\times\R^n.
	\]
	\begin{lemma}\label{Lem3.1}
		The space
		\[
		\mathscr C=C_c^\infty(G\setminus\{x=0\})
		\]
		is dense in $\dS(G)$.
	\end{lemma}
	\begin{proof}
		Fix $u\in C_c^\infty(G)$ and let $\chi_\varepsilon$ depend only on $r=\abs{x}$. Then $\abs{\nabla_G\chi_\varepsilon}=\abs{\nabla_x\chi_\varepsilon}$. If $m>2$, take $\chi_\varepsilon=0$ for $r\leq\varepsilon$ and $\chi_\varepsilon=1$ for $r\geq2\varepsilon$, with derivative bounded by $C/\varepsilon$. On the fixed support of $u$,
		\[
		\int_G\abs{u}^2\abs{\nabla_G\chi_\varepsilon}^2\dd g
		\leq C\varepsilon^{m-2}.
		\]
		If $m=2$, let $\chi_\varepsilon$ increase logarithmically from zero to one between $r=\varepsilon^2$ and $r=\varepsilon$. Its derivative is bounded by $C/(r\abs{\log\varepsilon})$, and the same integral is at most $C/\abs{\log\varepsilon}$. Smooth profiles with these bounds can be used. The term $(1-\chi_\varepsilon)\nabla_Gu$ converges to zero by dominated convergence. Thus $\chi_\varepsilon u\to u$ in energy. Density of $C_c^\infty(G)$ in $\dS(G)$ completes the proof.
	\end{proof}
	
	For $x\ne0$, let $\Pi_\ell$ denote the ordinary spherical harmonic projection of degree $\ell$ in $x/\abs{x}\in S^{m-1}$. Expansion of \eqref{eq1.2} gives
	\begin{equation}\label{eq3.1}
		\Delta_G=\Delta_x+s\Delta_t+
		\sum_{\alpha=1}^n\Omega_\alpha\partial_{t_\alpha},
		\qquad \Omega_\alpha=(J_\alpha x)\cdot\nabla_x.
	\end{equation}
	Each $\Omega_\alpha$ is a Euclidean rotation field and preserves harmonic degree. It follows that $\Delta_G$ and multiplication by $D^{-1}$ preserve every degree. In particular, distinct degrees are orthogonal for $\cE$ and $\cN$ on $\mathscr C$. The angular expansion of a function in $\mathscr C$ converges with the derivatives in \eqref{eq3.1} on its compact radial and central support. Its partial sums consequently converge in energy. Lemma~\ref{Lem3.1} gives orthogonal projections on $\dS(G)$ and
	\begin{equation}\label{eq3.2}
		u=\sum_{\ell=0}^\infty u_\ell,\qquad
		\cE(u)=\sum_{\ell=0}^\infty\cE(u_\ell),\qquad
		\cN(u)=\sum_{\ell=0}^\infty\cN(u_\ell),
		\qquad u_\ell=\Pi_\ell u.
	\end{equation}
	A weak solution of \eqref{eq1.8} therefore satisfies the same equation in each degree. This follows by testing against $\Pi_\ell\varphi$ for $\varphi\in\mathscr C$ and then using density. In particular,
	\begin{equation}\label{eq3.3}
		\cE(u_\ell)=\mu\cN(u_\ell).
	\end{equation}
	Thus preservation of horizontal harmonic degree gives the orthogonal decomposition needed for the weighted eigenvalue problem.
	
	To describe a fixed degree, identify homogeneous harmonic polynomials with symmetric trace-free tensors and write
	\begin{equation}\label{eq3.4}
		u_\ell(x,t)=\ip{F(s,t)}{x^{\otimes\ell}},
		\qquad F(s,t)\in\Sym_0^\ell(V),\qquad a=b+\ell.
	\end{equation}
	For $\ell=0$ the tensor space is $\R$. We use the Euclidean tensor inner product. Its angular integral is
	\begin{equation}\label{eq3.5}
		\int_{S^{m-1}}\ip{F}{\theta^{\otimes\ell}}
		\ip{H}{\theta^{\otimes\ell}}\dd\theta
		=c_\ell\ip{F}{H},\qquad
		c_\ell=\frac{\abs{S^{m-1}}\ell!}{m(m+2)\cdots(m+2\ell-2)}.
	\end{equation}
	For $\ell=0$, the denominator is the empty product, equal to one. Formula~\eqref{eq3.5} follows by integrating products of $2\ell$ coordinates: contractions within either tensor vanish by the trace-free condition, and the $\ell!$ pairings between the two tensors remain.
	
	For $1\leq r\leq\ell$, let $J_\alpha^{(r)}$ act on the $r$th factor of $V^{\otimes\ell}$ and set
	\[
	\cD_r=\sum_{\alpha=1}^nJ_\alpha^{(r)}\partial_{t_\alpha}.
	\]
	We regard each $\cD_r$ as an operator on $V^{\otimes\ell}$. Their sum preserves $\Sym_0^\ell(V)$. On $C_c^\infty(\R^n;V^{\otimes\ell})$, formal adjoints with respect to $\dd t$ and the Clifford relations give
	\begin{equation}\label{eq3.6}
		\cD_r^*=\cD_r,\qquad \cD_r^2=-\Delta_t\Id,
		\qquad
		\int_{\R^n}\abs{\cD_rF}^2\dd t
		=\int_{\R^n}\abs{\nabla_tF}^2\dd t.
	\end{equation}
	Here and below the norms involving an individual $\cD_r$ are full tensor norms.
	
	On $\HH$ define
	\[
	\begin{aligned}
		O_a&=-s(\partial_s^2+\Delta_t)-a\partial_s,\\
		\cB_a(F)&=\int_{\HH}s^a\bigl(\abs{F_s}^2+\abs{\nabla_tF}^2\bigr)\dd s\dd t,\\
		\cN_a(F)&=\int_{\HH}s^{a-1}D^{-1}\abs{F}^2\dd s\dd t.
	\end{aligned}
	\]
	The symbol $D$ in the half-space denotes $(1+s)^2+\abs{t}^2$. For a fixed homogeneous harmonic polynomial $H_\ell$ and a scalar coefficient $f(s,t)$,
	\[
	\Delta_x\bigl(H_\ell(x)f(s,t)\bigr)
	=H_\ell(x)\bigl(sf_{ss}+(b+\ell)f_s\bigr).
	\]
	Applying this identity to the tensor components of \eqref{eq3.4}, and applying the rotation fields in \eqref{eq3.1}, yields
	\begin{equation}\label{eq3.7}
		-\Delta_Gu_\ell
		=\ip{O_aF+\sum_{r=1}^\ell\cD_rF}{x^{\otimes\ell}}.
	\end{equation}
	The sign in the mixed term follows from $J_\alpha^*=-J_\alpha$ and the leading minus sign in $-\Delta_G$. Polar integration and \eqref{eq3.5} give
	\begin{equation}\label{eq3.8}
		\begin{gathered}
			\cE(u_\ell)=C_\ell\cE_\ell(F),\qquad
			\cN(u_\ell)=C_\ell\cN_a(F),\qquad
			C_\ell=2^{m+2\ell-1}c_\ell,\\
			\cE_\ell(F)=\cB_a(F)+
			\sum_{r=1}^\ell\int_{\HH}s^{a-1}\ip{F}{\cD_rF}\dd s\dd t.
		\end{gathered}
	\end{equation}
	The factor $C_\ell$ cancels in the weighted eigenvalue problem for each degree.
	
	\begin{lemma}\label{Lem3.2}
		For a compactly supported smooth field $F:\HH\to\Sym_0^\ell(V)$,
		\begin{equation}\label{eq3.9}
			\cE_\ell(F)=\frac1a\sum_{r=1}^\ell
			\int_{\HH}s^a\abs{F_s-\cD_rF}^2\dd s\dd t
			+\frac ba\cB_a(F).
		\end{equation}
		Furthermore,
		\begin{equation}\label{eq3.10}
			\frac ba\cB_a(F)\leq\cE_\ell(F)
			\leq\left(1+\frac\ell a\right)\cB_a(F).
		\end{equation}
		These statements extend to the energy completion of the degree-$\ell$ subspace.
	\end{lemma}
	\begin{proof}
		By \eqref{eq3.6} and integration by parts in $s$,
		\[
		\begin{aligned}
			\int_{\HH}s^a\abs{F_s-\cD_rF}^2\dd s\dd t
			&=\cB_a(F)-2\int_{\HH}s^a\ip{F_s}{\cD_rF}\dd s\dd t\\
			&=\cB_a(F)+a\int_{\HH}s^{a-1}\ip{F}{\cD_rF}\dd s\dd t.
		\end{aligned}
		\]
		Sum over $r$ and use $a-\ell=b$ to obtain \eqref{eq3.9}. Replacing the minus sign in the square by a plus sign gives the upper bound in \eqref{eq3.10}. Together with \eqref{eq3.8}, these bounds identify the degree-$\ell$ completion with the completion of compactly supported coefficient fields for $\cB_a^{1/2}$. Every square term is continuous in this norm by \eqref{eq3.6}. The norm equivalence is used for each fixed $\ell$.
		For $\ell\geq1$ one has $a>1$. Integration of $\partial_s(s^{a-1}|F|^2)$ followed by Cauchy-Schwarz gives
		\[
		\int_{\HH}s^{a-2}|F|^2\dd s\dd t
		\leq\frac{4}{(a-1)^2}\int_{\HH}s^a|F_s|^2\dd s\dd t.
		\]
		Together with the $L^2$ identity for $\cD_r$, this also makes each mixed integral absolutely convergent on the completed coefficient space.
	\end{proof}
	
	For a finite-dimensional Euclidean space $E$, let $\cY_a(E)$ be the completion of $C_c^\infty(\HH;E)$ in the norm $\cB_a^{1/2}$, and write $\cY_a$ in the scalar case. The following identity controls the weighted mass.
	
	\begin{lemma}\label{Lem3.3}
		Let $a\geq1$ and set
		\[
		w_a=D^{-(a+n-1)/2},\qquad \lambda_a=a(a+n-1).
		\]
		For $F\in\cY_a(E)$,
		\begin{equation}\label{eq3.11}
			\cB_a(F)-\lambda_a\cN_a(F)
			=\int_{\HH}s^aw_a^2
			\abs{\nabla_{s,t}(F/w_a)}^2\dd s\dd t\geq0.
		\end{equation}
		If the right-hand side vanishes, then $F=w_a C$ for a constant $C\in E$.
	\end{lemma}
	\begin{proof}
		Writing $A=1+s$, the Euclidean radial derivatives are
		\[
		\Delta_{s,t}w_a=a(a+n-1)D^{-1}w_a,
		\qquad (w_a)_s=-(a+n-1)A D^{-1}w_a.
		\]
		Since $A-s=1$, these identities give $O_aw_a=\lambda_aD^{-1}w_a$. Moreover,
		\[
		O_a=-s^{1-a}\operatorname{div}_{s,t}(s^a\nabla_{s,t}).
		\]
		Integration by parts proves \eqref{eq3.11} on the core. In particular, $\cN_a\leq\lambda_a^{-1}\cB_a$, and the right-hand side is a continuous quadratic form in the core norm. Completion extends the identity. On compact subsets of $\HH$, the limit quotient has the stated weak gradient. If that gradient vanishes, connectedness of $\HH$ gives $F/w_a=C$.
	\end{proof}
	
	Since $s^{a-1}D^{-1}$ is bounded below on compact subsets of $\HH$, the estimate $\cN_a\leq\lambda_a^{-1}\cB_a$ realizes each coefficient completion as a space of locally square-integrable fields. If a Cauchy sequence in $\cB_a^{1/2}$ has zero weighted $L^2$ limit, its derivative limit vanishes in distributions on every compact subset of $\HH$. The global weighted derivative limit is therefore zero. This proves that the realization is injective, and the weak derivatives and representation \eqref{eq3.4} pass to the completion.
	
	Combining Lemmas~\ref{Lem3.2} and \ref{Lem3.3} gives
	\begin{equation}\label{eq3.12}
		\begin{aligned}
			\cE_\ell(F)-b(a+n-1)\cN_a(F)
			&=\frac1a\sum_{r=1}^\ell\int_{\HH}s^a
			\abs{F_s-\cD_rF}^2\dd s\dd t\\
			&\quad+\frac ba\int_{\HH}s^aw_a^2
			\abs{\nabla_{s,t}(F/w_a)}^2\dd s\dd t.
		\end{aligned}
	\end{equation}
	The constant on its left satisfies
	\begin{equation}\label{eq3.13}
		b(a+n-1)=\mu_\star+b(\ell-2).
	\end{equation}
	
	\begin{proposition}\label{Prop3.4}
		If $u\in\dS(G)$ satisfies \eqref{eq1.8} with $\mu\leq\mu_\star$, then $\Pi_\ell u=0$ for every $\ell\geq2$.
	\end{proposition}
	\begin{proof}
		For $\ell>2$, equations \eqref{eq3.3}, \eqref{eq3.12} and \eqref{eq3.13} force $\cN_a(F)=0$. The same is true for $\ell=2$ when $\mu<\mu_\star$. If $\ell=2$ and $\mu=\mu_\star$, both nonnegative terms in \eqref{eq3.12} vanish. Lemma~\ref{Lem3.3} gives $F=w_a C$, and the first square gives $F_s-\cD_1F=0$. Substitution yields
		\[
		0=-\frac{(a+n-1)w_a}{D}
		\bigl((1+s)\Id-J_t^{(1)}\bigr)C.
		\]
		By \eqref{eq1.1},
		\[
		\bigl((1+s)\Id-J_t^{(1)}\bigr)^*
		\bigl((1+s)\Id-J_t^{(1)}\bigr)=D\Id.
		\]
		Thus $C=0$, which proves vanishing also in degree two.
	\end{proof}
	
	\section{The lower harmonic degrees}\label{Sec4}
	\subsection{The scalar ball operator}
	
	For the two remaining horizontal degrees, we use a Euclidean change of variables on the auxiliary half-space $(s,t)\in\HH$.
	
	Let $N=n+1$ and let $\BB\subset\R^N$ be the unit ball, with coordinates $y=(y_0,y')$, where $y'=(y_1,\ldots,y_n)$. Define $\Psi:\HH\to\BB$ by
	\begin{equation}\label{eq4.1}
		y_0=\frac{s^2+\abs{t}^2-1}{D},\qquad
		y_\alpha=\frac{2t_\alpha}{D},\quad 1\leq\alpha\leq n.
	\end{equation}
	If
	\[
	e(y)=(1-y_0)^2+\abs{y'}^2,\qquad h(y)=1-\abs{y}^2,
	\]
	then its inverse is
	\[
	s=\frac{h}{e},\qquad t=\frac{2y'}{e},\qquad D=\frac4e.
	\]
	In particular $\Psi$ is a diffeomorphism. Direct differentiation gives
	\begin{equation}\label{eq4.2}
		h=\frac{4s}{D},\qquad
		(D\Psi)^T(D\Psi)=\frac4{D^2}\Id,\qquad
		\abs{\det D\Psi}=\left(\frac2D\right)^N.
	\end{equation}
	For $a\geq1$, put
	\[
	\dd\nu_a=h^{a-1}\dd y,
	\qquad \cL_a=-h\Delta_y+2a\,y\cdot\nabla_y
	=-h^{1-a}\operatorname{div}(h^a\nabla_y).
	\]
	Write
	\[
	B_a(g,k)=\int_{\BB}h\ip{\nabla g}{\nabla k}\dd\nu_a,
	\qquad B_a(g)=B_a(g,g),
	\]
	and define $\cW_a(E)$ as the completion of $C_c^\infty(\BB;E)$ for
	\[
	\norm{g}_{\cW_a}^2=\norm{g}_{L^2(\nu_a)}^2+B_a(g).
	\]
	Again the target $E$ is omitted in scalar statements. This completion embeds injectively into $L^2(\nu_a)$: on each compact subset of the ball the weights are positive and bounded above and below, so a Cauchy sequence whose $L^2$ limit is zero has zero distributional gradient limit there. Exhausting the ball shows that its weighted gradient limit is zero everywhere. Thus $B_a$ is a densely defined closed form in $L^2(\nu_a)$, with the completed weak gradient as its form gradient.
	
	\begin{lemma}\label{Lem4.1}
		The map $g\mapsto w_a(g\circ\Psi)$ extends to an isomorphism from $\cW_a(E)$ onto $\cY_a(E)$. For $F=w_a(g\circ\Psi)$,
		\begin{equation}\label{eq4.3}
			\cN_a(F)=A_a\norm{g}_{L^2(\nu_a)}^2,
			\qquad
			\cB_a(F)=A_a\bigl[\lambda_a\norm{g}_{L^2(\nu_a)}^2+B_a(g)\bigr],
			\qquad A_a=2^{-N}4^{1-a}.
		\end{equation}
		On smooth functions the corresponding operator identity is
		\begin{equation}\label{eq4.4}
			O_a\bigl(w_a(g\circ\Psi)\bigr)
			=D^{-1}w_a\bigl[(\lambda_a+\cL_a)g\bigr]\circ\Psi.
		\end{equation}
	\end{lemma}
	\begin{proof}
		For $g\in C_c^\infty(\BB;E)$, the Jacobian in \eqref{eq4.2} gives the mass formula after cancellation of the powers of $D$. Applying the same identity to the last integral in \eqref{eq3.11} gives $A_aB_a(g)$ and proves \eqref{eq4.3}. Polarization gives \eqref{eq4.4}. The transformation maps the compactly supported smooth cores bijectively onto each other. Since $\lambda_a>0$, the norm equivalence in \eqref{eq4.3} extends this map to the completions.
	\end{proof}
	
	Every polynomial belongs to $\cW_a$. To prove this, for $a>1$, a cutoff in $h$ which vanishes for $h<\varepsilon$ and equals one for $h>2\varepsilon$ has additional gradient energy bounded by
	\[
	C\varepsilon^{-2}\int_0^{2\varepsilon}h^a\dd h
	\leq C\varepsilon^{a-1}.
	\]
	For $a=1$, use a logarithmic cutoff with derivative at most $C/(h\abs{\log\varepsilon})$ between $h=\varepsilon^2$ and $h=\varepsilon$. Since $\abs{\nabla h}$ and the density of Euclidean volume in the boundary variable $h$ are bounded there, its additional energy is bounded by
	\[
	\frac{C}{\abs{\log\varepsilon}^2}
	\int_{\varepsilon^2}^{\varepsilon}\frac{\dd h}{h}
	=\frac{C}{\abs{\log\varepsilon}}\longrightarrow0.
	\]
	The $L^2(\nu_a)$ error and the term involving the polynomial's gradient tend to zero by dominated convergence. Thus every polynomial, including the constants, belongs to $\cW_a$ for all $a\geq1$.
	
	Let $P_j^{(\alpha,\beta)}$ denote the Jacobi polynomial of degree $j$.
	
	\begin{lemma}\label{Lem4.2}
		For $a\geq1$, choose an orthonormal basis of spherical harmonics in each degree $k$, and let $H_k$ range over their homogeneous harmonic extensions. The polynomials
		\begin{equation}\label{eq4.5}
			H_k(y)P_j^{(a-1,k+N/2-1)}(2\abs{y}^2-1),
			\qquad k,j\geq0,
		\end{equation}
		form a complete orthogonal eigenbasis of the nonnegative self-adjoint operator associated with the closed form $B_a$ on $\cW_a$. Their eigenvalues are
		\begin{equation}\label{eq4.6}
			\lambda_{k,j}^{(a)}
			=2ak+4j\left(j+k+a+\frac N2-1\right).
		\end{equation}
		The operator has compact resolvent. After normalizing the eigenpolynomials in $L^2(\nu_a)$, its form expansion is obtained by multiplying each squared eigenbasis coefficient by \eqref{eq4.6}. Consequently, the weak nullspace of $\cL_a$ in $\cW_a$ consists of constants, and
		\begin{equation}\label{eq4.7}
			\ker_{\cW_a}(\cL_a-2a)
			=\Span\{y_0,\ldots,y_n\}.
		\end{equation}
	\end{lemma}
	\begin{proof}
		Let $z=\abs{y}^2$. Direct computation gives
		\[
		\begin{aligned}
			\cL_a\bigl(H_k(y)f(z)\bigr)
			=H_k(y)\bigl[&-4z(1-z)f''(z)\\
			&+4\bigl((a+k+N/2)z-(k+N/2)\bigr)f'(z)
			+2akf(z)\bigr].
		\end{aligned}
		\]
		The Jacobi equation with parameters $a-1$ and $k+N/2-1$ gives \eqref{eq4.5}-\eqref{eq4.6}. Both parameters exceed $-1$. Orthogonality follows from spherical harmonics and the one-variable Jacobi weight under $z=r^2$ \cite{DunklXu2014}.
		
		The decomposition of homogeneous polynomials into powers of $|y|^2$ times harmonic polynomials, followed by the triangular change to Jacobi polynomials, shows that \eqref{eq4.5} spans all polynomials. Polynomials are uniformly dense in continuous functions on the closed ball, and hence dense in $L^2(\nu_a)$. The cutoff argument above places each eigenpolynomial in $\cW_a$. Its equation, tested against compactly supported smooth functions, extends by continuity to every test function in $\cW_a$.
		
		The weak identities place every eigenpolynomial in the domain of the self-adjoint operator associated with $B_a$. Normalizing this orthogonal family gives a complete orthonormal eigenbasis in $L^2(\nu_a)$. The spectral theorem gives the operator and form expansions. Since the eigenvalues in \eqref{eq4.6} tend to infinity as $k+j\to\infty$ and have finite multiplicity, the operator has compact resolvent. Finite eigenpolynomial sums are also dense in the form norm.
		
		Testing a weak eigenfunction against this basis determines its expansion. Formula~\eqref{eq4.6} vanishes exactly at $(k,j)=(0,0)$ and equals $2a$ exactly at $(k,j)=(1,0)$. Indeed, for $j\geq1$ it is at least $4a+2N>2a$. This proves both kernel statements in $\cW_a$.
	\end{proof}
	
	The following estimate keeps track of the angular degree.
	
	\begin{lemma}\label{Lem4.3}
		If $g_k\in\cW_a(E)$ has Euclidean spherical harmonic degree $k$ in $y/\abs{y}$, then
		\begin{equation}\label{eq4.8}
			B_a(g_k)\geq2ak\norm{g_k}_{L^2(\nu_a)}^2.
		\end{equation}
	\end{lemma}
	\begin{proof}
		For a smooth single component write $g_k=r^kf(r)Y_k(\theta)$, with $\norm{Y_k}_{L^2(S^{N-1})}=1$. Integration of the radial cross term gives
		\[
		\begin{aligned}
			&\int_0^1(1-r^2)^a
			\left(\abs{(r^kf)'}^2
			+\frac{k(k+N-2)}{r^2}\abs{r^kf}^2\right)r^{N-1}\dd r\\
			&\quad=2ak\int_0^1(1-r^2)^{a-1}r^{2k}\abs{f}^2r^{N-1}\dd r
			+\int_0^1(1-r^2)^ar^{2k}\abs{f'}^2r^{N-1}\dd r.
		\end{aligned}
		\]
		At $r=0$, the boundary term vanishes for $k\geq1$ because $2k+N-2>0$, and its coefficient is zero for $k=0$. At $r=1$, it vanishes on the compactly supported core. Summation over a spherical harmonic basis proves the inequality for smooth functions. Angular projection is bounded in the scalar form norm, so the inequality extends to $\cW_a(E)$ by density.
	\end{proof}
	
	\subsection{The Clifford angular operator}
	
	For $V$-valued smooth functions on $\BB$, define
	\begin{equation}\label{eq4.9}
		\begin{aligned}
			\Gamma={}&-\sum_{\alpha=1}^nJ_\alpha
			(y_0\partial_{y_\alpha}-y_\alpha\partial_{y_0})-\sum_{\alpha<\beta}J_\alpha J_\beta
			(y_\alpha\partial_{y_\beta}-y_\beta\partial_{y_\alpha}).
		\end{aligned}
	\end{equation}
	The matrices in both sums are skew-adjoint. The rotation fields are also skew-adjoint for every radial measure, and they commute with the constant matrices. Hence $\Gamma$ is symmetric. It is tangential to Euclidean spheres and preserves each spherical harmonic degree.
	
	\begin{lemma}\label{Lem4.4}
		On the $V$-valued spherical harmonics of degree $k\geq1$, the eigenvalues of $\Gamma$ are $-k$ and $k+N-2$, with respective multiplicities
		\begin{equation}\label{eq4.10}
			m\binom{k+n-1}{n-1},\qquad
			m\binom{k+n-2}{n-1}.
		\end{equation}
		Both eigenspaces are nonzero. On constants, $\Gamma=0$ with multiplicity $m$.
	\end{lemma}
	\begin{proof}
		On the doubled module $V\oplus V$, define
		\[
		E_0=\begin{pmatrix}0&\Id\\-\Id&0\end{pmatrix},
		\qquad
		E_\alpha=\begin{pmatrix}0&J_\alpha\\J_\alpha&0\end{pmatrix}.
		\]
		They satisfy $E_i^*=-E_i$ and $E_iE_j+E_jE_i=-2\delta_{ij}\Id$. Define
		\[
		\widetilde\Gamma
		=-\sum_{0\leq i<j\leq n}E_iE_j
		(y_i\partial_{y_j}-y_j\partial_{y_i}),
		\qquad
		\mathscr D=\sum_{i=0}^nE_i\partial_{y_i},
		\qquad \mathbf y=\sum_{i=0}^nE_i y_i.
		\]
		The products are block diagonal, with
		\[
		(E_0E_\alpha)|_{V\oplus\{0\}}=J_\alpha,
		\qquad
		(E_\alpha E_\beta)|_{V\oplus\{0\}}=J_\alpha J_\beta.
		\]
		Thus the first copy of $V$ is invariant under $\widetilde\Gamma$, and its restriction there is \eqref{eq4.9}. Writing $\mathscr E=y\cdot\nabla_y$, anticommutation gives
		\begin{equation}\label{eq4.11}
			\mathscr D^2=-\Delta_y,\qquad
			\mathbf y\mathscr D=-\mathscr E-\widetilde\Gamma,
			\qquad
			\mathscr D\mathbf y+\mathbf y\mathscr D=-2\mathscr E-N.
		\end{equation}
		If $H$ is a $(V\oplus V)$-valued homogeneous harmonic polynomial of degree $k\geq1$, put
		\[
		N_{k-1}=-\frac{\mathscr DH}{2k+N-2},
		\qquad M_k=H-\mathbf yN_{k-1}.
		\]
		The denominator is positive. Equations~\eqref{eq4.11} show that
		\[
		\mathscr DN_{k-1}=0,\qquad \mathscr DM_k=0.
		\]
		They also give
		\[
		\widetilde\Gamma M_k=-kM_k,\qquad
		\widetilde\Gamma(\mathbf yN_{k-1})
		=(k+N-2)\mathbf yN_{k-1}.
		\]
		Thus every degree-$k$ harmonic polynomial is the sum of vectors in the two indicated eigenspaces. They are orthogonal because $\widetilde\Gamma$ is symmetric on the sphere. To obtain the multiplicities, first let $H$ take values in the first copy of $V$. Then $M_k$ takes values in that copy, while $N_{k-1}$ takes values in the second copy. A homogeneous polynomial $M$ satisfying $\mathscr DM=0$ is determined by its restriction $M(0,y')$ to $y_0=0$, since
		\[
		\partial_{y_0}M=E_0\sum_{\alpha=1}^nE_\alpha\partial_{y_\alpha}M.
		\]
		Conversely, any homogeneous degree-$k$ polynomial $A(y')$ valued in either fixed copy has the unique polynomial extension
		\[
		M(y_0,y')=\sum_{r=0}^k\frac{y_0^r}{r!}
		\left(E_0\sum_{\alpha=1}^nE_\alpha\partial_{y_\alpha}\right)^r A(y').
		\]
		The coefficient matrices are block diagonal, so the extension stays in the chosen copy of $V$. The space of these extensions has dimension $m\binom{k+n-1}{n-1}$. Each extension is harmonic because $\mathscr D^2=-\Delta_y$. Multiplication by $\mathbf y$ is injective because $\mathbf y^2=-|y|^2\Id$. Thus the space of polynomials $\mathbf yN_{k-1}$ taking values in the first copy of $V$ has dimension $m\binom{k+n-2}{n-1}$. The decomposition and the distinct eigenvalues give \eqref{eq4.10}. This is the monogenic decomposition for the angular Clifford operator \cite{BrackxDelangheSommen1982}.
	\end{proof}
	
	On compactly supported smooth functions set
	\[
	\gamma_a(g,k)=\int_{\BB}\ip{\Gamma g}{k}\dd\nu_a.
	\]
	The following bounds extend this form to $\cW_a(V)$.
	
	\begin{lemma}\label{Lem4.5}
		The form $\gamma_a$ extends to a continuous symmetric bilinear form on $\cW_a(V)$. For $g\in\cW_a(V)$,
		\begin{equation}\label{eq4.12}
			\gamma_a(g,g)\geq-\frac1{2a}B_a(g),\qquad
			\abs{\gamma_a(g,g)}
			\leq\frac1{2a}B_a(g)+(N-2)\norm{g}_{L^2(\nu_a)}^2.
		\end{equation}
		If $a>1$, the form
		\[
		Q_a(g,k)=B_a(g,k)+2\gamma_a(g,k)
		\]
		satisfies
		\begin{equation}\label{eq4.13}
			Q_a(g,g)\geq\left(1-\frac1a\right)B_a(g).
		\end{equation}
		For $a>1$, every $g\in\cW_a(V)$ such that $Q_a(g,k)=0$ for all compactly supported smooth $k$ is constant.
	\end{lemma}
	\begin{proof}
		Expand a smooth function in spherical harmonics, $g=\sum_{k\geq0}g_k$. Lemma~\ref{Lem4.4} and estimate~\eqref{eq4.8} of Lemma~\ref{Lem4.3} imply
		\[
		\gamma_a(g,g)\geq-\sum_{k\geq0}k\norm{g_k}_{L^2(\nu_a)}^2
		\geq-\frac1{2a}B_a(g).
		\]
		The upper angular eigenvalue gives the absolute bound in \eqref{eq4.12}. Angular convergence extends these estimates from finite sums to smooth functions. Polarization, followed by rescaling the two arguments, gives
		\[
		\abs{\gamma_a(g,k)}\leq C_a\norm{g}_{\cW_a}\norm{k}_{\cW_a}.
		\]
		Hence the bilinear form is continuous in the $\cW_a$ norm. This proves its extension and \eqref{eq4.13}. By continuity, the identity $Q_a(g,k)=0$ holds for every $k\in\cW_a(V)$. Taking $k=g$ gives $B_a(g)=0$ when $a>1$. Since $h>0$ inside the connected ball, $g$ is constant.
	\end{proof}
	
	For $g\in\cW_a(V)$, we interpret the angular term through the continuous form $\gamma_a$. This definition includes all the products $J_\alpha J_\beta$ in \eqref{eq4.9}.
	
	For horizontal degree one, define on $\HH$
	\[
	\begin{gathered}
		\cD=\sum_{\alpha=1}^nJ_\alpha\partial_{t_\alpha},\qquad
		\cO_a=O_a+\cD,\qquad
		q=(1+s)\Id+J_t,\\
		M_a=qD^{-(a+n)/2},\qquad
		\mu_a=(a-1)(a+n),\qquad a>1.
	\end{gathered}
	\]
	In particular,
	\begin{equation}\label{eq4.14}
		M_a^*M_a=w_a^2\Id,
	\end{equation}
	so $M_a$ is invertible on $\HH$.
	
	\begin{lemma}\label{Lem4.6}
		For $g\in C^\infty(\BB;V)$ and $F=M_a(g\circ\Psi)$,
		\begin{equation}\label{eq4.15}
			\cO_aF=D^{-1}M_a
			\bigl[(\mu_a+\cL_a+2\Gamma)g\bigr]\circ\Psi.
		\end{equation}
	\end{lemma}
	\begin{proof}
		Write $A=1+s$, $T=J_t$ and $\beta_a=(a+n)/2$. Thus $q=A\Id+T$, $q^{-1}=(A\Id-T)/D$ and $M_a=qD^{-\beta_a}$. Direct differentiation gives
		\[
		\begin{aligned}
			\Delta_{s,t}M_a&=\frac{2\beta_a(a-1)}D M_a,\\
			\partial_sM_a&=D^{-\beta_a}
			\left(\Id-\frac{2\beta_a A}Dq\right),\\
			\cD M_a&=D^{-\beta_a}
			\left(-n\Id-\frac{2\beta_a}DTq\right).
		\end{aligned}
		\]
		Since
		\[
		((a+s)\Id-T)q=D\Id+(a-1)q,
		\]
		substitution into $\cO_a$ gives
		\begin{equation}\label{eq4.16}
			\cO_aM_a=\mu_aD^{-1}M_a.
		\end{equation}
		This determines the zeroth-order term.
		
		For the first-order coefficients, write $M_a=w_aR$ with $R=q/\sqrt D$. Then
		\begin{equation}\label{eq4.17}
			\begin{aligned}
				R^{-1}R_s&=-\frac TD,\\
				R^{-1}R_{t_\alpha}
				&=\frac{AJ_\alpha-TJ_\alpha-t_\alpha\Id}{D},\\
				R^{-1}J_\alpha R
				&=\frac{(A^2-\abs{t}^2)J_\alpha
					+A(J_\alpha T-TJ_\alpha)+2t_\alpha T}{D}.
			\end{aligned}
		\end{equation}
		The coordinate derivatives in \eqref{eq4.1} are
		\begin{equation}\label{eq4.18}
			\begin{aligned}
				(y_0)_s&=\frac{2(A^2-\abs{t}^2)}{D^2},&
				(y_0)_{t_\alpha}&=\frac{4At_\alpha}{D^2},\\
				(y_\beta)_s&=-\frac{4At_\beta}{D^2},&
				(y_\beta)_{t_\alpha}
				&=\frac{2\delta_{\alpha\beta}}D-\frac{4t_\alpha t_\beta}{D^2}.
			\end{aligned}
		\end{equation}
		The scalar conjugation in \eqref{eq4.4} contributes $2ay\cdot\nabla_y$. By the product rule, the additional matrix coefficient of $\partial_{y_j}$ is
		\begin{equation}\label{eq4.19}
			\begin{aligned}
				C_j=D\biggl[&-2sR^{-1}R_s(y_j)_s
				-2s\sum_{\alpha=1}^nR^{-1}R_{t_\alpha}(y_j)_{t_\alpha}\\
				&+\sum_{\alpha=1}^nR^{-1}J_\alpha R(y_j)_{t_\alpha}\biggr].
			\end{aligned}
		\end{equation}
		Insert \eqref{eq4.17} and \eqref{eq4.18} and use
		\[
		T^2=-\abs{t}^2\Id,\qquad
		TJ_\alpha+J_\alpha T=-2t_\alpha\Id.
		\]
		The contractions in \eqref{eq4.19} satisfy
		\[
		\sum_{\alpha=1}^nt_\alpha R^{-1}R_{t_\alpha}=\frac{AT}{D},
		\qquad
		\sum_{\alpha=1}^nt_\alpha R^{-1}J_\alpha R=T.
		\]
		Consequently,
		\[
		C_0=\frac{4T}{D},\qquad
		C_\beta=\frac{-2(s^2+\abs{t}^2-1)J_\beta
			+4(J_\beta T+t_\beta\Id)}{D}.
		\]
		Since $J_\beta T+t_\beta\Id=\sum_{\alpha\ne\beta}t_\alpha J_\beta J_\alpha$, the resulting coefficients are
		\[
		C_0=2\sum_{\alpha=1}^ny_\alpha J_\alpha,
		\qquad
		C_\beta=-2y_0J_\beta+
		2\sum_{\alpha\ne\beta}y_\alpha J_\beta J_\alpha.
		\]
		They are precisely the coefficients of $2\Gamma$ in \eqref{eq4.9}. Finally, \eqref{eq4.2} gives the principal part $-h\Delta_y$. Together with \eqref{eq4.16}, these computations prove \eqref{eq4.15}.
	\end{proof}
	
	For $V$-valued coefficients on $\HH$, define
	\[
	E_a^{(1)}(F)=\cB_a(F)+
	\int_{\HH}s^{a-1}\ip{F}{\cD F}\dd s\dd t.
	\]
	This form is defined first on the compactly supported smooth core and then extended to $\cY_a(V)$. When $a=b+1$, this is the form $\cE_1$ in \eqref{eq3.8}. Applying the calculation in Lemma~\ref{Lem3.2} with one tensor factor gives, for every $a>1$,
	\begin{equation}\label{eq4.20}
		\left(1-\frac1a\right)\cB_a(F)
		\leq E_a^{(1)}(F)
		\leq\left(1+\frac1a\right)\cB_a(F).
	\end{equation}
	
	\begin{proposition}\label{Prop4.7}
		For $a>1$, the map $g\mapsto M_a(g\circ\Psi)$ extends to an isomorphism from $\cW_a(V)$ onto $\cY_a(V)$, with
		\begin{equation}\label{eq4.21}
			\begin{aligned}
				\cN_a(F)&=A_a\norm{g}_{L^2(\nu_a)}^2,\\
				E_a^{(1)}(F)&=A_a\left[
				\mu_a\norm{g}_{L^2(\nu_a)}^2+Q_a(g,g)\right].
			\end{aligned}
		\end{equation}
		It sends the weak equation $\cO_aF=\mu D^{-1}F$ to
		\begin{equation}\label{eq4.22}
			Q_a(g,k)=(\mu-\mu_a)(g,k)_{L^2(\nu_a)}
			\quad\text{for every }k\in\cW_a(V).
		\end{equation}
	\end{proposition}
	\begin{proof}
		For core functions, multiply identity~\eqref{eq4.15} from Lemma~\ref{Lem4.6} by $F$ and integrate against $s^{a-1}\dd s\dd t$. The left-hand side is $E_a^{(1)}(F)$. The mass transformation follows from \eqref{eq4.14} and the scalar Jacobian calculation. This gives \eqref{eq4.21}. Denote the polarized forms by the same symbols. For $H=M_a(k\circ\Psi)$ one obtains explicitly
		\[
		E_a^{(1)}(F,H)-\mu\cN_a(F,H)
		=A_a\bigl[Q_a(g,k)+(\mu_a-\mu)(g,k)_{L^2(\nu_a)}\bigr].
		\]
		
		By \eqref{eq4.12} and \eqref{eq4.13},
		\[
		\mu_a\norm{g}_2^2+
		\left(1-\frac1a\right)B_a(g)
		\leq \mu_a\norm{g}_2^2+Q_a(g,g)
		\leq \bigl(\mu_a+2(N-2)\bigr)\norm{g}_2^2
		+\left(1+\frac1a\right)B_a(g).
		\]
		Here the norms are in $L^2(\nu_a)$. Since $\mu_a>0$ and $1-1/a>0$, these bounds and \eqref{eq4.20} give equivalent completed form norms. The bijection of compactly supported smooth cores therefore extends to an isomorphism. The polarized identity, first applied to core tests and then extended by continuity, gives \eqref{eq4.22} on the full form domain.
	\end{proof}
	
	\begin{proof}[Proof of Theorem~\ref{Thm1.2}]
		Let $v$ satisfy \eqref{eq1.9}. Decompose it as in \eqref{eq3.2}. Proposition~\ref{Prop3.4} gives $v_\ell=0$ for $\ell\geq2$.
		
		For degree zero, write $v_0(x,t)=F(s,t)$. Equations \eqref{eq3.7} and \eqref{eq3.10} give
		\[
		O_bF=\mu_\star D^{-1}F,\qquad F\in\cY_b.
		\]
		Use $F=w_b(g\circ\Psi)$ in Lemma~\ref{Lem4.1}. Since $w_b=U$ and $\lambda_b=\kappa$, the weak equation becomes
		\[
		\cL_bg=(\mu_\star-\kappa)g=2bg.
		\]
		By \eqref{eq4.7}, $g$ is linear in $y$. Pulling back \eqref{eq4.1} yields
		\begin{equation}\label{eq4.23}
			v_0=D^{-(Q+2)/4}
			\left[c_0(1-\rho^4)+2\ip{\tau}{t}\right]
		\end{equation}
		after changing the sign of the coefficient of $y_0$.
		
		For degree one, write $v_1=\ip{x}{F(s,t)}$. Here $a=b+1>1$, so
		\[
		\cO_aF=\mu_\star D^{-1}F,\qquad F\in\cY_a(V).
		\]
		Apply Proposition~\ref{Prop4.7}. Since
		\[
		\mu_a=(a-1)(a+n)=b(b+n+1)=\mu_\star,
		\]
		the transformed equation is $Q_a(g,k)=0$ for every $k\in\cW_a(V)$. Lemma~\ref{Lem4.5} shows that $g=c$ is a constant vector. Consequently,
		\begin{equation}\label{eq4.24}
			\begin{aligned}
				v_1&=D^{-(Q+2)/4}\ip{x}{((1+s)\Id+J_t)c}\\
				&=D^{-(Q+2)/4}\ip{c}{(1+s)x-J_tx}.
			\end{aligned}
		\end{equation}
		Equations \eqref{eq4.23} and \eqref{eq4.24} prove the asserted classification.
		
		For the converse, differentiation of \eqref{eq1.6} at $(1,e)$ gives
		\begin{equation}\label{eq4.25}
			\begin{aligned}
				\left.\partial_\lambda U_{\lambda,e}\right|_{\lambda=1}
				&=\frac{Q-2}{2}D^{-(Q+2)/4}(1-\rho^4),\\
				\left.\frac{\dd}{\dd\varepsilon}U_{1,(0,\varepsilon\tau)}
				\right|_{\varepsilon=0}
				&=\frac{Q-2}{2}D^{-(Q+2)/4}\ip{\tau}{t},\\
				\left.\frac{\dd}{\dd\varepsilon}U_{1,(\varepsilon c,0)}
				\right|_{\varepsilon=0}
				&=\frac{Q-2}{4}D^{-(Q+2)/4}
				\ip{c}{(1+s)x-J_tx}.
			\end{aligned}
		\end{equation}
		Differentiating \eqref{eq1.4} along these parameters proves \eqref{eq1.9}. The functions are smooth on $G$, decay at least as $\rho^{2-Q}$, and their horizontal gradients are $O(\rho^{1-Q})$ at infinity. Homogeneous polar coordinates show that these gradients are square integrable at infinity. Multiplication by a smooth cutoff on $B_{2R}$ equal to one on $B_R$ gives an additional gradient term with squared norm $O(R^{2-Q})$, which tends to zero. Thus these functions belong to the homogeneous completion $\dS(G)$.
		
		The numerators in \eqref{eq1.10} are linearly independent. Their constant terms determine $c_0$, their horizontal linear terms determine $c$, and their central linear terms determine $\tau$. Hence the kernel has dimension $m+n+1$.
	\end{proof}
	
	The same conclusion can be expressed as a finite-energy Liouville theorem for an operator with polynomial coefficients. Define
	\[
	\cT P=-D\Delta_GP+\frac{Q+2}{2}\ip{\nabla_GD}{\nabla_GP}
	-\frac{Q+2}{4}\abs{x}^2P.
	\]
	\begin{corollary}\label{Cor4.8}
		If $\cT P=0$ in distributions and $D^{-(Q+2)/4}P\in\dS(G)$, then
		\[
		P=c_0(1-\rho^4)+2\ip{\tau}{t}+\ip{c}{(1+s)x-J_tx}
		\]
		for some $c_0\in\R$, $c\in V$ and $\tau\in Z$. In particular, $P$ is a polynomial of weighted degree at most four, with $\deg x_i=1$ and $\deg t_\alpha=2$.
	\end{corollary}
	
	\begin{proof}[Proof of Corollary~\ref{Cor4.8}]
		Put $\beta=(Q+2)/4$. The product rule and \eqref{eq2.2} give
		\[
		\bigl(-\Delta_G-\mu_\star D^{-1}\bigr)(D^{-\beta}P)
		=D^{-\beta-1}\cT P.
		\]
		Since $D>0$, all coefficients and multipliers are smooth, and the identity holds in distributions. Apply Theorem~\ref{Thm1.2} to $D^{-\beta}P$ and multiply by $D^\beta$. The stated polynomial formula follows.
	\end{proof}
	
	\section{Spectral consequences and global stability}\label{Sec5}
	We use the notation
	\[
	L_U=-\Delta_G-\mu_\star D^{-1},\qquad
	\cZ=\ker_{\dS(G)}L_U,
	\]
	and use $\perp_{\dS}$ for orthogonality in the energy inner product $\cE(\cdot,\cdot)$.
	
	\begin{proposition}\label{Prop5.1}
		The weighted problem \eqref{eq1.8} has discrete positive spectrum. Its first eigenvalue is $\kappa$, simple with eigenfunction $U$. Its second distinct eigenvalue is $\mu_\star$, with eigenspace $\cZ$ and multiplicity $m+n+1$. The quadratic form of $L_U$ has Morse index one.
	\end{proposition}
	\begin{proof}
		By Lemma~\ref{Lem2.1}, the bounded operator $K$ on $\dS(G)$ defined by
		\[
		\cE(Ku,v)=\cN(u,v)
		\]
		is compact, self-adjoint and positive. It is injective because $D^{-1}>0$. The compact spectral theorem gives an eigenbasis of $K$ that is orthonormal with respect to $\cE$. Its eigenvalues are positive and tend to zero. Their reciprocals are the eigenvalues of \eqref{eq1.8}, with the same eigenspaces.
		
		The ground-state identity gives the simple first eigenvalue $\kappa$. Let $v$ be an eigenfunction with eigenvalue at most $\mu_\star$. Proposition~\ref{Prop3.4} reduces $v$ to horizontal degrees zero and one. In degree zero, Lemmas~\ref{Lem4.1} and \ref{Lem4.2} give the first two values $\lambda_b=\kappa$ and $\lambda_b+2b=\mu_\star$. In degree one, Proposition~\ref{Prop4.7} and \eqref{eq4.13} give the lower bound $\mu_a=\mu_\star$. Thus the spectrum up to $\mu_\star$ consists of $\kappa$ and $\mu_\star$. Theorem~\ref{Thm1.2} gives the multiplicity of $\mu_\star$.
		
		The energy Riesz representative of $L_U$ is $\Id-\mu_\star K$. On a weighted eigenvector of eigenvalue $\mu$ it acts by $1-\mu_\star/\mu$. Exactly one such value is negative, corresponding to $\mu=\kappa$.
	\end{proof}
	
	\subsection{Spectral bounds beyond the symmetry modes}
	
	The subspaces of horizontal harmonic degrees zero and one admit a complete spectral description. In the following formulas, $k$ denotes harmonic degree in the auxiliary variable $y\in\R^{n+1}$, whereas the horizontal degree in $x$ is fixed at zero or one. We write $\mathscr H_k(\R^{n+1})$ for the space of homogeneous harmonic polynomials of degree $k$. If different pairs or branches give the same numerical eigenvalue, their multiplicities are added.
	
	\begin{proposition}\label{Prop5.2}
		In horizontal degree zero, the eigenvalues of \eqref{eq1.8} are
		\begin{equation}\label{eq5.1}
			\mu^{(0)}_{k,j}
			=\kappa+2bk+4j\left(j+k+b+\frac{n-1}{2}\right),
			\qquad k,j\geq0,
		\end{equation}
		with multiplicity
		\[
		\dim\mathscr H_k(\R^{n+1})
		=\binom{k+n}{n}-\binom{k+n-2}{n}.
		\]
		A binomial coefficient with upper entry smaller than its nonnegative lower entry is taken to be zero. In horizontal degree one the two eigenvalue branches are
		\begin{align}
			\mu^{(1,-)}_{k,j}
			&=\mu_\star+2bk
			+4j\left(j+k+b+\frac{n+1}{2}\right),
			&&k,j\geq0,\label{eq5.2}\\
			\mu^{(1,+)}_{k,j}
			&=\mu_\star+2(b+2)k+2(n-1)
			+4j\left(j+k+b+\frac{n+1}{2}\right),
			&&k\geq1,\ j\geq0,\label{eq5.3}
		\end{align}
		with multiplicities $m\binom{k+n-1}{n-1}$ and $m\binom{k+n-2}{n-1}$, respectively. In each of these two horizontal sectors, the least eigenvalue strictly larger than $\mu_\star$ is $\mu_\star+m$.
	\end{proposition}
	\begin{proof}
		For horizontal degree zero, Lemma~\ref{Lem4.1} takes the weighted eigenvalue problem to $\cL_bg=(\mu-\kappa)g$. The complete form expansion in Lemma~\ref{Lem4.2} gives \eqref{eq5.1}. The harmonic dimension follows by subtracting the dimension of degree-$(k-2)$ polynomials from that of degree-$k$ polynomials.
		
		For horizontal degree one, put $a=b+1$ in Proposition~\ref{Prop4.7}. The problem becomes
		\[
		(\cL_a+2\Gamma)g=(\mu-\mu_\star)g.
		\]
		In each $V$-valued harmonic space, choose the orthogonal eigenbasis of $\Gamma$ from Lemma~\ref{Lem4.4}. The operator $\Gamma$ annihilates radial scalar functions. Multiplication by the Jacobi factor in \eqref{eq4.5} therefore diagonalizes $\cL_a$ and $\Gamma$ simultaneously. Adding $2(-k)$ or $2(k+n-1)$ to \eqref{eq4.6} gives \eqref{eq5.2}-\eqref{eq5.3}, with multiplicities \eqref{eq4.10}. The products span the vector-valued polynomials, a dense subspace of $L^2(\nu_a;V)$. By Lemma~\ref{Lem4.5}, the quadratic form $Q_a(g,g)+\norm{g}_{L^2(\nu_a)}^2$ is equivalent to $\norm{g}_{\cW_a}^2$. The associated operator consequently has this complete eigenbasis.
		
		In degree zero, $(k,j)=(0,0)$ and $(1,0)$ give $\kappa$ and $\mu_\star$. The next value, at $(2,0)$, is $\kappa+4b=\mu_\star+m$. Every $j\geq1$ gives a larger value. In degree one, $(0,0)$ in the minus branch gives $\mu_\star$, and $(1,0)$ gives $\mu_\star+2b$. The plus branch starts at $\mu_\star+2b+2n+2$. Every positive radial index also gives a larger value.
	\end{proof}
	
	We next strengthen the bound supplied by \eqref{eq3.12}, with particular attention to horizontal degree two. Recall the coefficient forms $\cE_\ell$ and $\cN_a$ from Section~\ref{Sec3}.
	
	\begin{lemma}\label{Lem5.3}
		Let $\ell\geq1$, $a=b+\ell$ and $L=a+n-1$. For every $F\in\cY_a(\Sym_0^\ell V)$,
		\begin{equation}\label{eq5.4}
			\cE_\ell(F)-bL\cN_a(F)
			\geq\delta_\ell\cN_a(F),
			\qquad
			\delta_\ell=\frac{2b\ell L}{2b+\ell(3L+4)}.
		\end{equation}
		In particular, for $\ell=2$ the right-hand coefficient is
		\begin{equation}\label{eq5.5}
			\delta_2=\frac{2b(b+n+1)}{4b+3n+7}
			=\frac{2\mu_\star}{2m+3n+7}.
		\end{equation}
	\end{lemma}
	\begin{proof}
		Write $F=F_0+F_\perp$, where $F_0=w_aC$ is the $\cN_a$-orthogonal projection onto the scalar ground-state subspace with constant tensors $C\in\Sym_0^\ell V$. This subspace belongs to $\cY_a$ by Lemma~\ref{Lem4.1} and the admissibility of constants on the ball. The ground-state equation implies
		\[
		\cB_a(F_0,F_\perp)=aL\cN_a(F_0,F_\perp)=0.
		\]
		Set
		\[
		X^2=\cB_a(F_\perp)-aL\cN_a(F_\perp),
		\quad \alpha^2=\cN_a(F_0),
		\quad S=\sum_{r=1}^{\ell}\int_{\HH}s^a|F_s-\cD_rF|^2\dd s\dd t.
		\]
		Under Lemma~\ref{Lem4.1}, $F_\perp$ corresponds to a vector-valued function on the ball with zero weighted mean. The first positive scalar eigenvalue is $2a$. Consequently,
		\begin{equation}\label{eq5.6}
			\cN_a(F_\perp)\leq\frac{X^2}{2a},
			\qquad
			\cB_a(F_\perp)\leq\frac{L+2}{2}X^2.
		\end{equation}
		By \eqref{eq3.12} and the orthogonal decomposition above,
		\begin{equation}\label{eq5.7}
			R_\ell:=\cE_\ell(F)-bL\cN_a(F)
			=\frac Sa+\frac ba X^2.
		\end{equation}
		
		For $P_r=\partial_s-\cD_r$, skew-adjointness of the matrices gives
		\[
		\sum_{r=1}^{\ell}\norm{P_rF_0}_{L^2(s^a)}^2
		=\ell\cB_a(F_0)=\ell aL\alpha^2.
		\]
		The notation $L^2(s^a)$ includes integration in $(s,t)$ and the full tensor norm. Since
		\[
		\int_{\R^n}|\cD_rH|^2\dd t
		=\int_{\R^n}|\nabla_tH|^2\dd t,
		\]
		the inequality $|A-B|^2\leq2|A|^2+2|B|^2$ and \eqref{eq5.6} yield
		\[
		\sum_{r=1}^{\ell}\norm{P_rF_\perp}_{L^2(s^a)}^2
		\leq2\ell\cB_a(F_\perp)
		\leq\ell(L+2)X^2.
		\]
		The triangle inequality in the product Hilbert space now gives
		\[
		\sqrt{\ell aL}\,\alpha
		\leq\sqrt S+\sqrt{\ell(L+2)}X.
		\]
		Cauchy's inequality in the two variables $\sqrt S$ and $\sqrt bX$, followed by \eqref{eq5.7}, implies
		\[
		\alpha^2
		\leq\left(\frac1{\ell L}+\frac{L+2}{bL}\right)R_\ell.
		\]
		Also \eqref{eq5.6} and \eqref{eq5.7} imply
		$\cN_a(F_\perp)\leq R_\ell/(2b)$. Adding these estimates proves
		\[
		\cN_a(F)\leq
		\left(\frac1{\ell L}+\frac{L+2}{bL}+\frac1{2b}\right)R_\ell.
		\]
		Taking the reciprocal of the coefficient proves \eqref{eq5.4}. Every identity extends from the smooth core by the form bounds in Section~\ref{Sec3}. Taking $\ell=2$ gives \eqref{eq5.5}.
	\end{proof}
	
	\begin{proposition}\label{Prop5.4}
		Let $\mu_{\mathrm{next}}$ be the next distinct weighted eigenvalue after $\mu_\star$. Then \eqref{eq1.11} holds, and
		\begin{equation}\label{eq5.8}
			\cE(w)-\mu_\star\cN(w)
			\geq\left(1-\frac{\mu_\star}{\mu_{\mathrm{next}}}\right)\cE(w)
			\quad\text{if } w\perp_{\dS}\bigl(\Span\{U\}+\cZ\bigr).
		\end{equation}
		One also has the explicit inverse estimate
		\begin{equation}\label{eq5.9}
			\norm{w}_{\dS}\leq\frac{2m+3n+9}{2}\norm{L_Uw}_{(\dS)^*}
			\quad\text{if }w\perp_{\dS}\cZ.
		\end{equation}
		Both estimates hold with the same constants at every member of $\cM$, after transporting the corresponding orthogonality conditions.
	\end{proposition}
	\begin{proof}
		Compactness gives a next distinct eigenvalue. In horizontal degrees zero and one, Proposition~\ref{Prop5.2} gives a gap $m=2b$ above $\mu_\star$. Lemma~\ref{Lem5.3} gives the gap $\delta_2$ in degree two. For $\ell\geq3$, \eqref{eq3.12}-\eqref{eq3.13} give a gap at least $b(\ell-2)\geq b$. Since $0<\delta_2<b$, orthogonal decomposition proves the lower bound in \eqref{eq1.11} on the complement of $U$ and $\cZ$.
		
		For the upper bound, take $H_2(y)=y_0^2-y_1^2$. The corresponding degree-zero eigenfunction is
		\[
		v_2(x,t)=D^{-(Q+6)/4}\bigl[(\rho^4-1)^2-4t_1^2\bigr],
		\]
		with eigenvalue $\kappa+4b=\mu_\star+m$ by Proposition~\ref{Prop5.2}. Lemma~\ref{Lem4.1} places it in the energy space, so it gives the upper bound in \eqref{eq1.11}. Expansion in the eigenbasis of $K$, orthogonal with respect to $\cE$, proves \eqref{eq5.8}. On $\cZ^{\perp_{\dS}}$, the self-adjoint Fredholm operator $\Id-\mu_\star K$ is injective and boundedly invertible. Under the Riesz isometry between $\dS$ and its dual, an admissible inverse constant is
		\[
		C_G=\max\left\{\frac{\kappa}{\mu_\star-\kappa},
		\left(1-\frac{\mu_\star}{\mu_{\mathrm{next}}}\right)^{-1}\right\}.
		\]
		The lower bound in \eqref{eq1.11} gives $(1-\mu_\star/\mu_{\mathrm{next}})^{-1}\leq(2m+3n+9)/2$. Also, $\kappa/(\mu_\star-\kappa)=(Q-2)/4$ is smaller than this number, proving \eqref{eq5.9}. The amplitude direction has negative linearized energy. Thus \eqref{eq5.8} is stated on the complement of both $U$ and $\cZ$, whereas the inverse estimate holds on $\cZ^{\perp_{\dS}}$.
		
		The maps $\cU_{\lambda,\eta}$ in \eqref{eq1.5} are unitary in the energy norm. The transformed weight is
		\[
		U_{\lambda,\eta}^{4/(Q-2)}(g)
		=\lambda^2D^{-1}\bigl(\delta_\lambda(\eta^{-1}g)\bigr).
		\]
		A change of variables therefore intertwines the weighted forms and the linearized operators. It also transports the energy orthogonal complements. This proves the uniformity along $\cM$.
	\end{proof}
	
	Put
	\[
	c_{\mathrm{lin}}=1-\frac{\mu_\star}{\mu_{\mathrm{next}}},\qquad
	\frac{2}{2m+3n+9}\leq c_{\mathrm{lin}}\leq\frac4{Q+6}.
	\]
	The local parameter family of nonzero extremals has, at $U$, the tangent space
	\[
	T_U\mathfrak M=\Span\{U\}\oplus\cZ,\qquad
	\dim T_U\mathfrak M=m+n+2.
	\]
	The sum is orthogonal with respect to $\cE$ because $U$ and $\cZ$ belong to different weighted eigenspaces. The normalized solution family has tangent space $T_U\cM=\cZ$ of dimension $m+n+1$. Both $\Def$ and $d^2$ are homogeneous of degree two and invariant under \eqref{eq1.5}. Since $0\in\mathfrak M$, we also have $d(u)\leq\norm u_{\dS}$.
	
	\subsection{Sharp local asymptotics and global stability}
	
	We first establish existence and local uniqueness of the nearest bubble, together with the regularity of its parameters.
	
	\begin{lemma}\label{Lem5.5}
		The nonzero part $\mathfrak M\setminus\{0\}$ of the extremal cone is locally a $C^2$ embedded manifold in $\dS(G)$. If $d(u)<\norm u_{\dS}$, the distance is attained at a nonzero bubble. There is $\varepsilon_G>0$ such that, when $d(u)<\varepsilon_G\norm u_{\dS}$, this nearest point is unique. Writing it as $z=cU_{\lambda,\eta}$, one has
		\begin{equation}\label{eq5.10}
			u=z+w,\qquad
			w\perp_{\dS}T_z\mathfrak M,\qquad
			d(u)=\norm w_{\dS},\qquad
			\cE(u)=c^2\cE(U)+d(u)^2.
		\end{equation}
		In a fixed neighborhood of $U$, the nearest-point parameters are $C^1$ functions of $u$. For every sufficiently small $r\perp_{\dS}T_U\mathfrak M$, the nearest point to $U+r$ is exactly $U$.
	\end{lemma}
	\begin{proof}
		Write $\theta=(\log\lambda,\eta)$ in exponential coordinates, so that $\theta=0$ corresponds to $(\lambda,\eta)=(1,e)$, and set $U_\theta=U_{\lambda,\eta}$. On a compact set of these parameters, the denominator in $U_{\lambda,\eta}$ is comparable to $D$. It is a polynomial of weighted degree at most four in $(x,t)$, and each horizontal derivative has degree at most three. Parameter derivatives up to order two preserve these bounds. The product rule gives, uniformly on the parameter set,
		\begin{equation}\label{eq5.11}
			\left|\partial_\theta^\beta U_\theta(g)\right|
			\leq C(1+\rho(g))^{2-Q},\qquad
			\left|\nabla_G\partial_\theta^\beta U_\theta(g)\right|
			\leq C(1+\rho(g))^{1-Q},\qquad |\beta|\leq2.
		\end{equation}
		The upper bound for the horizontal gradients in \eqref{eq5.11} belongs to $L^2(G)$ because $Q>2$. The first bound and the cutoff argument in Lemma~\ref{Lem2.1} place these parameter derivatives in $\dS(G)$. Dominated convergence then proves that the orbit is $C^2$ in energy. The tangent vectors in \eqref{eq4.25}, together with the amplitude derivative, are linearly independent by Theorem~\ref{Thm1.2} and Proposition~\ref{Prop5.1}.
		
		Next, if $(\lambda_j,\eta_j)$ escapes every compact parameter set, then
		\begin{equation}\label{eq5.12}
			U_{\lambda_j,\eta_j}\rightharpoonup0\quad\text{in }\dS(G).
		\end{equation}
		It suffices to test against $\varphi\in C_c^\infty(G)$ and use
		$\cE(U_{\lambda,\eta},\varphi)=\kappa\int U_{\lambda,\eta}^{2^*-1}\varphi$.
		If $\lambda_j\to0$, the global bound
		$U_{\lambda_j,\eta_j}\leq\lambda_j^{(Q-2)/2}$ gives convergence to zero. If $\lambda_j\to\infty$, then
		\[
		\int_G U_{\lambda_j,\eta_j}^{2^*-1}\dd g
		=\lambda_j^{-(Q-2)/2}\int_G U^{2^*-1}\dd g\longrightarrow0.
		\]
		The integral is finite since $U^{2^*-1}=O(\rho^{-Q-2})$. If $\lambda_j$ stays in a compact subset of $(0,\infty)$, then $\rho(\eta_j)\to\infty$ along an escaping subsequence, and the bubble tends uniformly to zero on the support of $\varphi$. These cases cover a subsequence of every escaping sequence. The bounded energy norm and density of the tests prove \eqref{eq5.12}.
		
		Optimizing first in the amplitude gives
		\begin{equation}\label{eq5.13}
			d(u)^2=\cE(u)-\sup_{\lambda,\eta}
			\frac{\cE(u,U_{\lambda,\eta})^2}{\cE(U)}.
		\end{equation}
		If $d(u)<\norm u_{\dS}$, the supremum in \eqref{eq5.13} is positive. By \eqref{eq5.12}, a maximizing sequence stays in a compact parameter set. Continuity gives a maximizing pair and a nonzero optimizing amplitude. Differentiation proves the orthogonality in \eqref{eq5.10}, and the amplitude direction gives the Pythagorean identity. The same argument shows that the cone is closed. Indeed, strong convergence to a nonzero limit bounds the amplitudes above and away from zero, while \eqref{eq5.12} bounds the parameters. Fixing the amplitude at one proves closedness of $\cM$.
		
		To identify the nearest point near $U$, suppose $c_jU_{\lambda_j,\eta_j}\to U$ strongly. Then $|c_j|\to1$, and \eqref{eq5.12} bounds the parameters in a compact set. Every parameter limit satisfies $cU_{\lambda,\eta}=U$. Positivity and the unique maximum of the bubble give $c>0$ and $\eta=e$. Evaluation at the identity and at $(0,t)$ gives $\lambda=1$ and $c=1$. Hence all parameters tend to $(1,1,e)$. Nearest points to functions converging to $U$ therefore have parameters converging to $(1,1,e)$.
		
		In local coordinates $\vartheta$ for the full cone, centered at $z_0=U$, apply the implicit function theorem to
		\[
		\cE\bigl(u-z_\vartheta,\partial_{\vartheta_i}z_\vartheta\bigr)=0.
		\]
		At $u=z_0=U$ the parameter derivative is minus the positive definite Gram matrix of the tangent vectors. There is a unique local stationary parameter, depending $C^1$ on $u$. Every global nearest point for $u$ sufficiently close to $U$ is in this parameter neighborhood, so it equals that stationary point. If $u=U+r$ with $r\perp_{\dS}T_U\mathfrak M$, then $\vartheta=0$ solves the stationary system. Thus $U$ is the unique nearest point to $u$.
		
		Finally normalize $\norm u_{\dS}=1$ and suppose $d(u)<\varepsilon_G$. Choose $c_0U_{\lambda_0,\eta_0}\in\mathfrak M$ with $\norm{u-c_0U_{\lambda_0,\eta_0}}_{\dS}<2\varepsilon_G$. Then $|c_0|\norm U_{\dS}\geq1-2\varepsilon_G$. After applying its inverse similarity and dividing by $c_0$, the distance to $U$ is at most
		$2\varepsilon_G\norm U_{\dS}/(1-2\varepsilon_G)$.
		For small $\varepsilon_G$ the fixed local argument applies. Similarities and nonzero scalar multiplication preserve uniqueness of minimizers. This proves uniqueness whenever $d(u)<\varepsilon_G\norm u_{\dS}$.
	\end{proof}
	
	\begin{proposition}\label{Prop5.6}
		Set $\sigma=\min\{1,4/(Q-2)\}$. There are $\varepsilon_G,C_G>0$ such that, for
		$0<d(u)<\varepsilon_G\norm u_{\dS}$ and $t_u=d(u)/\norm u_{\dS}$,
		\begin{equation}\label{eq5.14}
			\Def(u)\geq\bigl(c_{\mathrm{lin}}-C_Gt_u^\sigma\bigr)d(u)^2.
		\end{equation}
		The optimal asymptotic local coefficient is exactly $c_{\mathrm{lin}}$:
		\begin{equation}\label{eq5.15}
			\lim_{\varepsilon\downarrow0}
			\inf_{\substack{u\in\dS(G)\\0<d(u)\leq\varepsilon\norm u_{\dS}}}
			\frac{\Def(u)}{d(u)^2}=c_{\mathrm{lin}}.
		\end{equation}
		In particular, after decreasing $\varepsilon_G$, the lower bound
		$\Def(u)\geq c_{\mathrm{lin}}d(u)^2/2$ holds in this neighborhood. A local inequality with right-hand side $c\,d(u)^q\norm u_{\dS}^{2-q}$ and $c>0$ requires $q\geq2$.
	\end{proposition}
	\begin{proof}
		Put $p_c=2^*$ and $A_U=\int_GU^{p_c}\dd g$. The second derivative of the deficit at $U$ is
		\[
		\tfrac12\Def''(U)[r,r]
		=\cE(r)-\mu_\star\cN(r)
		+\frac{\kappa(p_c-2)}{A_U}
		\left(\int_GU^{p_c-1}r\dd g\right)^2.
		\]
		If $r\perp_{\dS}U$, the last term vanishes by \eqref{eq1.4}.
		For $2<p_c\leq3$, the inequality
		\[
		\big||v|^{p_c-2}-|z|^{p_c-2}\big|\leq C|v-z|^{p_c-2}
		\]
		gives H\"older continuity of the multiplication part of the Hessian, as a map from $L^{p_c}$ to $L^{p_c/(p_c-2)}$. For $p_c\geq3$, the corresponding estimate follows from
		\[
		\big||v|^{p_c-2}-|z|^{p_c-2}\big|
		\leq C(|v|+|z|)^{p_c-3}|v-z|
		\]
		and H\"older's inequality. The remaining Hessian terms are products of first derivatives of the norm and smooth scalar powers of $\int|v|^{p_c}$, which stays bounded away from zero near $U$. They are locally Lipschitz. The critical embedding thus makes $\Def''$ locally H\"older continuous in energy with exponent $\sigma$. Taylor's formula with integral remainder proves
		\begin{equation}\label{eq5.16}
			\left|\Def(U+r)-\bigl(\cE(r)-\mu_\star\cN(r)\bigr)\right|
			\leq C_G\norm r_{\dS}^{2+\sigma},
			\qquad r\perp_{\dS}U,
		\end{equation}
		for every sufficiently small real $r$ satisfying the stated orthogonality.
		
		Use the decomposition \eqref{eq5.10} from Lemma~\ref{Lem5.5}, and set
		$r=c^{-1}\cU_{\lambda,\eta}^{-1}w$. It is normal to $T_U\mathfrak M$. The exact distance and Pythagorean identities give
		\begin{equation}\label{eq5.17}
			d(u)=|c|\norm r_{\dS},\qquad
			\norm r_{\dS}=\norm U_{\dS}\frac{t_u}{\sqrt{1-t_u^2}}.
		\end{equation}
		Homogeneity and invariance of the deficit, \eqref{eq5.8}, and \eqref{eq5.16} yield \eqref{eq5.14}. This proves the lower bound in \eqref{eq5.15}.
		
		For the reverse bound, choose an eigenfunction of unit energy $\psi$ at $\mu_{\mathrm{next}}$. It is orthogonal to $U$ and to $\cZ$. For small real $\epsilon$, Lemma~\ref{Lem5.5} gives
		\[
		d(U+\epsilon\psi)=|\epsilon|,
		\qquad
		\Def(U+\epsilon\psi)=c_{\mathrm{lin}}\epsilon^2
		+O(|\epsilon|^{2+\sigma}).
		\]
		This proves \eqref{eq5.15}. For the same perturbation, the quotient
		\[
		\frac{\Def(U+\epsilon\psi)}
		{d(U+\epsilon\psi)^q\norm{U+\epsilon\psi}_{\dS}^{2-q}}
		\]
		tends to zero as $\epsilon\to0$ when $q<2$, proving optimality of the exponent. This local expansion corresponds to the Euclidean argument in \cite[Lemma~1]{BianchiEgnell1991}.
	\end{proof}
	
	The preceding asymptotics also identify the normal directions along which the local coefficient is approached. Let $E_{\mathrm{next}}$ denote the full eigenspace at $\mu_{\mathrm{next}}$, and let $\mu_{\mathrm{after}}$ be the next distinct eigenvalue above it. Set
	\[
	c_{\mathrm{after}}=1-\mu_\star/\mu_{\mathrm{after}}>c_{\mathrm{lin}}.
	\]
	
	\begin{corollary}\label{Cor5.7}
		Let $u$ satisfy $0<d(u)<\varepsilon_G\norm u_{\dS}$, as in Proposition~\ref{Prop5.6}. Use the decomposition \eqref{eq5.10} and the transported normal vector $r$ in \eqref{eq5.17}. Put $q_u=r/\norm r_{\dS}$. Then
		\begin{equation}\label{eq5.18}
			\dist_{\dS}(q_u,E_{\mathrm{next}})^2
			\leq
			\frac{\Def(u)/d(u)^2-c_{\mathrm{lin}}+C_Gt_u^\sigma}
			{c_{\mathrm{after}}-c_{\mathrm{lin}}}.
		\end{equation}
		Consequently, for a sequence of such functions with $t_{u_j}=d(u_j)/\norm{u_j}_{\dS}\to0$,
		the equivalence
		\[
		\frac{\Def(u_j)}{d(u_j)^2}\longrightarrow c_{\mathrm{lin}}
		\quad\Longleftrightarrow\quad
		\dist_{\dS}(q_{u_j},E_{\mathrm{next}})\longrightarrow0
		\]
		holds.
	\end{corollary}
	\begin{proof}
		Expand the normal vector $q_u$, of unit energy, in the weighted eigenbasis. Its components in $E_{\mathrm{next}}$ have linearized energy $c_{\mathrm{lin}}$, while every remaining component has linearized energy at least $c_{\mathrm{after}}$. Thus
		\[
		\cE(q_u)-\mu_\star\cN(q_u)
		\geq c_{\mathrm{lin}}
		+(c_{\mathrm{after}}-c_{\mathrm{lin}})
		\dist_{\dS}(q_u,E_{\mathrm{next}})^2.
		\]
		By \eqref{eq5.16} and \eqref{eq5.17}, this linearized energy differs from $\Def(u)/d(u)^2$ by at most $C_Gt_u^\sigma$. This proves \eqref{eq5.18} and the forward implication. Conversely, the bounded operator $\Id-\mu_\star K$ acts as $c_{\mathrm{lin}}\Id$ on $E_{\mathrm{next}}$, so convergence of the normal direction to that eigenspace implies convergence of its linearized energy to $c_{\mathrm{lin}}$. The Taylor error tends to zero as well.
	\end{proof}
	
	\begin{proof}[Proof of Theorem~\ref{Thm1.1}]
		Theorem~\ref{Thm2.6} gives the equality classification. For the global estimate, argue by contradiction and choose $u_j$ with $d(u_j)>0$ such that
		\[
		\frac{\Def(u_j)}{d(u_j)^2}\longrightarrow0.
		\]
		By homogeneity normalize $\cE(u_j)=1$. Since $d(u_j)\leq1$, we have $\Def(u_j)\to0$ and $\norm{u_j}_{2^*}^2\to S_G^{-1}$. Thus $u_j/\norm{u_j}_{2^*}$ is a normalized minimizing sequence. By Theorem~\ref{Thm2.7}, similarities transform a subsequence into one converging strongly to a scalar multiple of $U$. Invariance of the cone and of its energy distance gives $d(u_j)\to0$. Proposition~\ref{Prop5.6} now yields
		\[
		\frac{\Def(u_j)}{d(u_j)^2}\geq\frac{c_{\mathrm{lin}}}{2}
		\]
		for all sufficiently large $j$, a contradiction.
	\end{proof}
	
	\subsection{Local rigidity and error estimates for the critical equation}
	
	The inverse estimate gives a local error bound for the normalized critical equation. We use the solution family $\cM$ from \eqref{eq1.6} and define
	\[
	\cF(u)=-\Delta_Gu-\kappa|u|^{4/(Q-2)}u
	\quad\text{as an element of }(\dS(G))^*.
	\]
	
	\begin{theorem}\label{Thm5.8}
		There exist $\varepsilon_G>0$ and $C_G\geq1$ such that
		\begin{equation}\label{eq5.19}
			C_G^{-1}\dist_{\dS}(u,\cM)
			\leq\norm{\cF(u)}_{(\dS)^*}
			\leq C_G\dist_{\dS}(u,\cM)
		\end{equation}
		for every real $u\in\dS(G)$ with $\dist_{\dS}(u,\cM)<\varepsilon_G$. In particular, every such solution of
		\begin{equation}\label{eq5.20}
			-\Delta_Gu=\kappa|u|^{4/(Q-2)}u
		\end{equation}
		belongs to $\cM$.
	\end{theorem}
	\begin{proof}
		Set $p=(Q+2)/(Q-2)$. The Sobolev and H\"older inequalities give $\cF\in C^1(\dS,(\dS)^*)$, with
		\[
		\ip{\cF'(u)w}{\varphi}
		=\cE(w,\varphi)-\kappa p\int_G|u|^{p-1}w\varphi\dd g.
		\]
		The derivative is bounded uniformly on a fixed energy ball, so $\cF$ is Lipschitz there. Since all members of $\cM$ have the same energy norm and satisfy $\cF(U_{\lambda,\eta})=0$, taking the infimum over nearby bubbles proves the upper bound in \eqref{eq5.19}.
		
		For the lower bound, first work near $U$. Write
		\[
		R(w)=\kappa\bigl(|U+w|^{p-1}(U+w)-U^p-pU^{p-1}w\bigr).
		\]
		The scalar power inequalities and H\"older's inequality yield
		\begin{equation}\label{eq5.21}
			\norm{R(w)}_{(\dS)^*}
			\leq
			\begin{cases}
				C\norm w_{\dS}^{p},&1<p\leq2,\\
				C\bigl(\norm w_{\dS}^{2}+\norm w_{\dS}^{p}\bigr),&p\geq2.
			\end{cases}
		\end{equation}
		Indeed, the pointwise remainders are bounded by $C|w|^p$ in the first range and by $C(U^{p-2}|w|^2+|w|^p)$ in the second. Each bound belongs to $L^{(p+1)/p}$ because $p+1=2^*$.
		
		Using the normalized parameters $\theta$ from Lemma~\ref{Lem5.5} and the $C^2$ regularity in \eqref{eq5.11}, we apply the implicit function theorem to
		\[
		\cE(u-U_\theta,\partial_{\theta_i}U_\theta)=0.
		\]
		The parameter derivative at $(U,0)$ is minus the Gram matrix of a basis of $\cZ$. We obtain $u=U_\theta+w_\theta$ with $w_\theta$ orthogonal to the tangent space of $\cM$. Transporting by $\cU_{\lambda,\eta}^{-1}$ gives $U+w$ with $w\perp_{\dS}\cZ$ and $\norm w_{\dS}$ small. The change of variables gives the dual equivariance identity
		\[
		\ip{\cF(\cU_{\lambda,\eta}v)}{\cU_{\lambda,\eta}\varphi}
		=\ip{\cF(v)}{\varphi}.
		\]
		Here $(Q-2)(p+1)/2=Q$, so the residual dual norm is invariant. Thus
		\[
		L_Uw=\cF(U+w)+R(w).
		\]
		The inverse estimate \eqref{eq5.9} and \eqref{eq5.21}, with the nonlinear term absorbed for sufficiently small $w$, imply
		\[
		\norm w_{\dS}\leq C_G\norm{\cF(U+w)}_{(\dS)^*}.
		\]
		Since $\dist_{\dS}(u,\cM)\leq\norm w_{\dS}$, this proves the lower bound near $U$. For a function close to $\cM$, choose a bubble within twice the distance threshold and transport it to $U$. Equivariance gives the same constants. When the distance is zero, closedness of $\cM$ from Lemma~\ref{Lem5.5} gives $u\in\cM$. This proves \eqref{eq5.19} throughout the stated neighborhood. Setting $\cF(u)=0$ proves the last assertion.
	\end{proof}
	
	Theorem~\ref{Thm5.8} shows that every solution sufficiently close to $\cM$ belongs to $\cM$ and that the residual norm is locally equivalent to the distance to this family. Theorem~\ref{Thm1.1} gives a global stability estimate for the Sobolev inequality relative to the extremal cone $\mathfrak M$.
	A natural question is whether every positive
	solution $u\in\dS(G)$ of \eqref{eq5.20} belongs to $\cM$
	for an arbitrary H-type group $G$. This is the finite energy
	classification problem of Garofalo and Vassilev
	\cite{GarofaloVassilev2001,Yang2024}, expressed in the normalization
	used here. Theorem~\ref{Thm2.6} establishes the classification
	for Sobolev extremals, while Theorem~\ref{Thm5.8} establishes
	it in an energy neighborhood of $\cM$.
	
	A second question concerns the optimal global stability constant
	\[
	C_{\mathrm{BE}}(G)
	=
	\inf_{\substack{u\in\dS(G)\\d(u)>0}}
	\frac{\Def(u)}{d(u)^2}.
	\]
	Theorem~\ref{Thm1.1} and Proposition~\ref{Prop5.6} imply
	\[
	0<C_{\mathrm{BE}}(G)
	\leq c_{\mathrm{lin}}
	=1-\frac{\mu_\star}{\mu_{\mathrm{next}}}.
	\]
	For a general H-type group, is this infimum attained, and does
	the strict inequality $C_{\mathrm{BE}}(G)<c_{\mathrm{lin}}$ hold?
	These questions extend the study of optimal stability constants
	in the Euclidean and Heisenberg settings
	\cite{Konig2023,Konig2025,TangZhangZhang2024}
	to arbitrary orthogonal Clifford modules.
	
	\section*{Acknowledgments}
	The author thanks Professor Ingo Witt for his supervision and Professor Thomas Schick for his support during his doctoral studies at the University of G\"ottingen. The author acknowledges the use of AI tools. All mathematical arguments and proofs in the final manuscript were written and checked by the author.

	\medskip
	{\bf Funding:} 
	This work is supported by the National Natural Science Foundation of China (12301145, 12561020, 12261107) and the Yunnan Fundamental Research Projects (202401AU070123, 202601AT070048). 
	
	\medskip
	{\bf Data availability:}  Data sharing is not applicable to this article as no new data were created or analyzed in this study.
	
	\medskip
	{\bf Conflict of interest:} The author declares no conflict of interest.
	
	\begingroup
	\small
	\setlength{\parskip}{0pt}
	\interlinepenalty=10000
	
	\endgroup
\end{document}